\documentclass{article}
\usepackage[left=1.5in,textwidth=30pc,textheight=54.5pc]{geometry}
\usepackage{amsfonts, amsmath, amssymb, amsthm, mathtools, relsize}
\usepackage{bm}
\usepackage{hyperref}
\usepackage{graphicx, color}
\usepackage[normalem]{ulem} 
\usepackage{fullpage}
\usepackage{cite}
\usepackage{enumerate}
\usepackage{graphicx}
\usepackage{xcolor}
\usepackage{tikz, tkz-graph} 
\usetikzlibrary{arrows}
\usepackage{authblk}
\usepackage{comment}

\usepackage[linewidth=1pt]{mdframed}

\usepackage[textwidth=1.4in]{todonotes}

\reversemarginpar
\usepackage{caption}

\definecolor{nb}{rgb}{.6, .176, 1}
\definecolor{sienna}{rgb}{1, 0, 0}
\definecolor{darkgreen}{rgb}{0, .5, 0}

\numberwithin{equation}{section}
\allowdisplaybreaks

\newtheorem{theorem}{Theorem}[section]
\newtheorem{proposition}[theorem]{Proposition}
\newtheorem{corollary}[theorem]{Corollary}

\newtheorem{lemma}[theorem]{Lemma}
\newtheorem{example}[theorem]{Example}

\newtheorem{remark}{Remark}
\newtheorem{defi}[theorem]{Definition}

\newtheorem{problem}[theorem]{Problem}

\newcommand{\GG}{\mbox{${\mathcal G}$}}
\newcommand{\FF}{\mbox{${\mathcal F}$}}

\newcommand{\OO}{\mbox{${\mathcal O}$}}

\newcommand{\DD}{\mbox{${\mathcal D}$}}
\newcommand{\sDD}{\tiny{\mbox{${\mathcal D}$}}}

\newcommand{\Rbold}{\mbox{${\mathbb R}$}}

\newcommand{\Z}{\mbox{${\mathbb Z}$}}

\newcommand{\bN}{\mbox{${\mathbb N}$}}
\newcommand{\I}{\mbox{${\mathbb I}$}}

\usepackage{amsfonts}

 \newcommand{\R}{\mathbb{R}}
 \makeatletter
\def\namedlabel#1#2{\begingroup  
    (#2)%
    \def\@currentlabel{#2}%
    \phantomsection\label{#1}\endgroup
}

\usepackage{cancel}

\usepackage{ulem}

\newcommand{\bS}{\mathbf{S}}
\newcommand{\bM}{\mathbf{M}}
\newcommand{\br}{\mathbf{r}}
\newcommand{\bs}{\mathbf{s}}

\newcommand{\bone}{\mathbf{1}}
\newcommand{\PP}{\mbox{${\mathbb P}$}}

\newcommand{\E}{\mbox{${\mathbb E}$}}

\title{Tampered Memory Elephant Random Walk on one-dimensional integer lattice}
\author{Vinita Mulay\footnote{{University of Duisburg-Essen}}~\footnote{\href{mailto:vinitamulay5@gmail.com}{vinita.mulay@uni-due.de}}
\and 
Neeraja Sahasrabudhe\footnote{{Indian Institute of Science Education and  Research, Mohali}}~\footnote{\href{mailto:neeraja@iisermohali.ac.in}{neeraja@iisermohali.ac.in}}
\and Debleena Thacker\footnote{{Durham University}}~\footnote{\href{mailto:debleena.thacker@durham.ac.uk}{debleena.thacker@durham.ac.uk}}}

\date{\today}

\begin{document}
\maketitle
\begin{abstract} 
One of the outstanding questions in the theory of elephant random walks as observed in\cite{gut2023elephant}, is to determine how much memory is needed for a phase transition between the diffusive, critical and superdiffusive regimes to persist. In an attempt to understand this memory breakpoint, we introduce the tampered memory elephant random walk, in which the memory is partitioned into two disjoint sets $D_n$ and $D_n^c$, which may be deterministic or random. On $D_n^c$ the dynamics is the same as an elephant random walk, while on $D_n$ the increments are replaced by independent innovations. The resulting walk is thus driven by two competing components: a memory-dependent elephant random walk and an independent classical simple random walk corresponding to the innovations. To the best of our knowledge, our model is novel and is reminiscent of the unresolved linear setting considered in \cite{Ro_Ta_Ta_2025}. We first establish a law of large numbers when the increasing collections $\{D_n\}_{n \ge 1}$ and $\{D^c_n\}_{n \ge 1}$ have a renewal structure with exponential moments. We then identify a sharp threshold that governs the persistence of the phase transition for deterministic memory partitions. More precisely, we show that if $\{D_n\}_{n \ge 1}$ is non-random and $\{D_n\}_{n \ge 1}$, and $\{D^c_n\}_{n \ge 1}$ are both increasing sequences with $\lim_{n \to \infty} \frac{\lvert D^c_n\rvert}{n} >1/2$, then a phase transition into diffusive, critical and superdiffusive regimes is still observable, whereas for $\lim_{n \to \infty} \frac{\lvert D^c_n \rvert}{n}<1/2$, there is only the diffusive regime with a scaling of $\OO(\sqrt{n})$. The case of $\lim_{n\to \infty}\frac{\lvert D^c_n \rvert}{n}=1/2$ is also completely characterised. Thus, one-half emerges as the sharp breakpoint for the persistence of anomalous diffusion in this competitive setting. We conjecture that the same threshold governs the case when $\{D_n\}_{n \ge 1}$ is random. Our proofs rely on stochastic approximation applied to the two dependent competing components of the walk, corresponding to the retained memory and the innovations. 

\end{abstract}
\section{Introduction and main results}
\label{Sec:Intro}
The Elephant Random Walk (ERW), or as we will sometimes refer to it as the \textit{classical} elephant random walk, was introduced in \cite{Sc_Tr_2004} to understand and investigate the impact of long-range memory on random walks. The ERW differs from the random walks in the sense that the steps or increments are not independent, but chosen from the past uniformly at random, and then it is either repeated or flipped with some probability independent of everything else. More formally: 
let $\left(\tilde{S}_n\right)_{n \ge 0}$ denote a classical elephant random walk (ERW) on the one-dimensional integer lattice, with $\tilde{S}_0=0$, and $\tilde{S}_n$ is defined recursively (see \cite{Ba_Be_2016,Be_2018}) as follows
\begin{equation}
\label{Def: ERW}
\tilde{S}_{n+1}=\tilde{S}_n + \tilde{X}_{n+1},
\end{equation}
where $\tilde{X}_{n}$ takes values in $\{ \pm 1\}$, such that, $P(\tilde{X}_1 = 1)=q \in [0, 1]$. The law of $\tilde{X}_n$ is determined by the following mechanism  for $0 \le p \le 1$, 
\begin{equation}
\label{Eq:Choice_of_step-sizes}
\tilde{X}_{n+1}= \begin{cases}
           \tilde{X}_{I_{n+1}}, \, & \text{ with probability } \ p, \\
           -\tilde{X}_{I_{n+1}}, \, & \text{ with probability } 1-p,
\end{cases}
\end{equation}
where $I_{n+1}$ is random variable taking values chosen uniformly from $\{1,2,\ldots, n\}$. A similar definition for ERW on $\Z^d$, for $d \ge 2$, can be found in \cite{Be_La_2019a}. The model that we will be interested in this paper will differ from the classical ERW in the sense that, in our model, we will tamper with a portion of the memory. 

It is well known from \cite{Sc_Tr_2004} that the ERW demonstrates anomalous diffusion. The mathematically rigorous results in ERW first appeared in \cite{Be_2018, Ba_Be_2016}. In \cite{Ba_Be_2016} the authors prove a coupling of the ERW with generalised Friedman urn(see\cite{Fre_65, Ja_2004} for urn models) and established a functional central limit theorem, while \cite{Be_2018} established, using martingale methods, that the ERW exhibits a phase transition into diffusive, critical and superdiffusive regimes with scalings $\OO(\sqrt{n})$, $\OO(\sqrt{n \log n})$ and $\OO(n^{2p-1})$, respectively for $0 \le p < 3/4$, $p=\frac{3}{4}$ and $3/4 < p \le 1$. Similar results are also found in \cite{coletti2017central}. Furthermore, it is also shown in \cite{Be_2018} that the limiting distribution in the superdiffusive case is non-Gaussian (see also \cite{Kubota2019Gaussian} for Gaussian fluctuations around this limiting distribution, and \cite{Gu_La_Ra_Si_2025, Gu_La_Ra_2026} for more details on the limiting distribution in the superdiffusive case). The corresponding phase transition results for the multidimensional version are shown in \cite{Be_La_2019a}, and it is known that the critical value is given by $p_d = \frac{2d+1}{4d}$ in $\Z^d$. Other interesting properties, such as the centre of mass (see \cite{Be_2018}) and the number of zeroes, have also been studied \cite{Ber_2022}. In \cite{Ber_2022}, the author proved that on $\Z$ the ERW is positive recurrent iff $p < \frac{1}{4}$, and conjectured that the ERW is transient in dimensions $3$ and above, which was settled in affirmation in \cite{Qin_2025}. Furthermore, in \cite{Qin_2025} it is also known that on $\Z^d$, $d=1,2$, the ERW exhibits a phase transition between recurrence and transience at the critical $p_d = \frac{2d+1}{4d}$, while for $d=1$, it was independently shown by \cite{Gu_La_Ra_2026}. Other interesting variants include ERW with amnesia in \cite{La_2022}, ERW with multiple elephants having graph-based interactions \cite{Da_2024}, and \cite{Ma_Ro_Sa_2025} with multiple memory channels, and \cite{Po_Ro_2026} where the steps are sampled according to some function $f$. A lazy version of ERW has been studied in \cite{bercu2022_lazy}. Another interesting model is that of \cite{gonzalez2021reinforced}, where an independent Bernoulli random variable determines whether the next increment will be determined by a past step sampled from memory or will be according to an independent random variable.

A variant of the classical ERW model is when the elephant has only partial memory, that is, when $I_{n+1}$ is chosen uniformly from a subset $\mathcal{M}_n$ of $\{1,2,\ldots, n\}$, see (\cite{Gu_St_2021, Gu_St_2022, gut2023elephant}). For an overview of ERW with partial memory and open problems, see \cite{gut2023elephant}. The partial memory models have also been studied numerically in \cite{Cr_da_Si_2007, daSilva_Cr_Vi_2005}.
In \cite{Gu_St_2021}, the authors observed that for limited memory, that is, when the elephant remembers only the distant $K=1,2$ steps, equivalently, $\mathcal{M}_n = \{1, 2, \ldots, K\}$, the ERW has diffusive regime only with a scaling of $\OO(\sqrt{n})$. \cite{Gu_St_2021} also considers the case for most recent memory, that is, when $\mathcal{M}_n= \{n \}$ or $\mathcal{M}_n= \{n-1, n\}$, and shows that there is no superdiffusive regime in this case as well. Similar results for the general $K$ finite memory case can be found in \cite{ben2021finite}. Since \cite{Gu_St_2021,ben2021finite} demonstrate that there is no phase transition in the finite memory cases, while it is known that there is a phase transition into superdiffusive regime when we sample from the entire memory, a natural open question (also discussed in \cite{Gu_St_2021}) is what is the breakpoint in memory that is needed to observe the phase transition into all three regimes. In \cite{Gu_St_2022}, another variant of partial memory is studied.
More recently, \cite{Ro_Ta_Ta_2025} further study the model in \cite{Gu_St_2022}, and consider two settings for increasing memory, namely the \textit{triangular array} and the \textit{linear} setting. The main results, including the phase transition into the superdiffusive regime obtained in \cite{Ro_Ta_Ta_2025}, are for the triangular array setting, and there are no results available for the linear setting, which is when $\mathcal{M}_n = \{1,2,\ldots, m_n\}$, where $ 1 \le m_n \le n$.

The model discussed in this paper (see Subsection \ref{SubSec: TMERW} for more details) with tampered memory is somewhat similar to the unresolved linear setting discussed in \cite{Ro_Ta_Ta_2025}. In our model, while still sampling from the entire memory, we tamper with a portion of memory where the increments are no longer as in the classical ERW, but are according to an innovation. We prove a law of large numbers and under certain assumptions, we obtain the \textit{breakpoint} in memory that ensures a phase transition into diffusive, critical and superdiffusive regimes (see Theorems \ref{Thm: SLLN_non-nested}, \ref{Thm: SLLN_renewal_chain}, \ref{Thm: CLT_D_n/n_goes_to_zero} and \ref{Thm: anamolous_diffusion_Dn_non_trivial}) for this memory tampered version of the ERW.

\subsection
{Our Model: Tampered Memory Elephant Random Walk (TMERW)}
\label{SubSec: TMERW}
We study a tampered memory elephant random walk (TMERW), wherein we model a random walk that behaves like a classical elephant random walk (ERW) if the memory point is chosen from a given subset of the past, but behaves differently when it is chosen from its complement. 

In preparation for the definition of TMERW, consider a sequence of sets $\mathcal{D} = \{ D_n \}_{n \geq 1}$, where $D_n \subset \{1, 2, \dots, n\}$. On the set $D_n$ the elephant has no memory of the past, and hence if the chosen past step is from $D_n$, then the increment is just an innovation. We denote these innovations by the sequence $\{Y_n\}_{n \ge 2}$ and assume that they are independent of everything else. More formally, we define  the following parameters
\begin{itemize}
\item $\mathcal{D} = \{ D_n \}_{n \geq 1}$, where $D_n \subset \{1, 2, \dots, n\}$ (which may be deterministic or random) such that, $D_n \cup D_n^c = \{1,2, \ldots, n\}$;
\item $\{Y_{n+1}\}_{n \ge 1}$ is a sequence of i.i.d.\,random variables, where $\PP\left(Y_n=1\right)=1-\PP\left( Y_n =-1\right)=\lambda$, for some $0 \le \lambda \le 1$. We will refer to $Y_n$ as the \textit{innovations}.
\item We assume that $\mathcal{D}$ and $\{Y_{n+1}\}_{n \ge 1}$ are independent of each other.
\end{itemize}

 Let $X_1$ be a random variable taking values $\pm 1$. Let $S_n$ denote the position of the Tampered Memory Elephant Random Walk (TMERW) taking values in $\Z$. Then $S_n$ is defined recursively as follows:  
\[
S_0 \equiv 0, S_1 \equiv X_1, 
\]
and for $n \ge 2$, having constructed $S_n$, to obtain $S_{n+1}$, we choose a number $I_{n+1}$ uniformly at random from the past $\{1,2,\ldots, n\}$. Then, 
\begin{equation}
\label{Def: TERW}
S_{n+1}=S_n + X_{n+1},
\end{equation}
where $X_{n+1}$ is defined by 
\begin{equation}
\label{Eq:Choice_of_step-sizes}
X_{n+1}= \begin{cases}
           X_{I_{n+1}}, \, & \text{ w.p. } \ p, \text{ if } I_{n+1} \notin D_n, \\
           -X_{I_{n+1}}, \, & \text{ w.p. } \ 1-p, \text{ if } I_{n+1} \notin D_n, \\
           Y_{n+1}, \,  & \text{ if }  I_{n+1} \in D_n.
\end{cases}
\end{equation}
Observe that if $D_n = \emptyset$, then we retrieve the classical ERW. A motivation for TMERW is to imagine a competition between the two walks, namely the simple random walk and the ERW, in which the walker decides the move according to whether the past step is chosen from $D_n$ or not. The idea of introducing innovations on the set $D_n$ is motivated by an attempt to understand the influence of \textit{memory} on the fluctuations of the elephant random walk and possibly have a slightly better understanding of the breakpoint of memory to observe the phase transition into the superdiffusive regime for the classical ERW, which is an open question as observed in \cite{Gu_St_2021}. Given the costs of storing large amounts of memory, it is quite natural to ask how much memory is actually necessary for desirable properties, and whether we get away with inducing some noise in the memory. 
\begin{remark}
A multidimensional version of TMERW can be defined with appropriate modifications in \eqref{Eq:Choice_of_step-sizes} similar to \cite{Be_La_2019a}. The technique of the proofs of the main theorems should also carry forward.
\end{remark}
We make some further assumptions on the structure of the sets $\mathcal{D}$, which are discussed below.

\begin{defi}[Nested Sequence of Memory Sets]
\label{Def: nested_sets}
The sequence $\mathcal{D}$ is called nested if both the collection $\{D_n\}_{n \ge 1}$ and $\{D^c_n\}_{n \ge 1}$ are non-decreasing, that is, $D_n \subseteq D_{n+1}$ and $D^c_n \subseteq D^c_{n+1}$ for all $n$. Otherwise, we will refer to the collection $\mathcal{D}$ as non-nested.
\end{defi}

For our purpose, the nested sequence $\mathcal{D}$ will be generated using an appropriate stochastic process that is independent of the walk $S_n$.
\begin{defi}[Collection of nested sequences $\mathcal{D}$ generated by a stochastic chain]
Consider a stochastic chain $\{L_n\}_{n \ge 1}$ taking values in $\{0, 1\}$, such that $L_1 \equiv 1$.
Setting $D_1^c \equiv 1$, we define $D_n^c$ recursively as 
\begin{equation}
\label{Eq: Def_D_n_c}
D^c_{n+1}= \begin{cases}
D^c_n \cup \{n+1\}, & \text{ if } L_{n+1}=1,\\
D^c_n, & \text{ otherwise.}
\end{cases}
\end{equation}
In this case, we refer to $\mathcal{D}$ as the collection of nested sequences generated by the chain $\{L_n\}_{n \ge 1}$. 
\end{defi}
A trivial observation is that if $\mathcal{D}$ is generated by the chain $\{L_n\}_{n \ge 1}$, then by the definition itself, $\{D^c_{n}\}_{n \ge 1}$ is an increasing sequence of sets. Furthermore, since $D_n \cup D_n^c = \{1,2,\ldots, n\}$, it is obvious that $D_1 =\emptyset$ and 
\begin{equation}
\label{Eq: Def_D_n}
D_{n+1}= \begin{cases}
D_n \cup \{n+1\}, & \text{ if } L_{n+1}=0,\\
D_n, & \text{ otherwise.}
\end{cases}
\end{equation}
Thus, the sequence $\{D_n\}_{n \ge 1}$ are non-decreasing, and this shows that $\mathcal{D}$ is a nested sequence. 
Before we give some typical examples of $\{L_n\}_{n \ge 1}$, we discuss a little more about the structure of the set $\DD$.
Recall that $L_1=1$. Now consider the following sequence of random times defined by \begin{equation}
\label{Eq: Def_tau_1}
\tau_1 := \inf \{n \ge 1: L_{n+1}=0\}; 
\end{equation}
where $\tau_1= \infty$ when the corresponding set is empty. 
If $\tau_1< \infty$, then 
\begin{equation}
\label{Eq: Def_sigma_1}
\sigma_1 := \inf \{n > \tau_1: L_{n+1}=1\} -\tau_1,  
\end{equation} with the convention that $\sigma_1= \infty$, when the above set is empty.
Otherwise, when $\tau_1= \infty$, we define 
$\sigma_1 \equiv 0$.
Furthermore, setting $\mu_0 \equiv 0$, define
\begin{equation*}
\mu_1= \tau_1+\sigma_1.
\end{equation*}
We can now define the collection $\tau_k$ and $\sigma_k$ recursively as follows. Setting $\mu_k:= \sum_{j=1}^k (\tau_j+\sigma_j)$, if $\mu_k<\infty$, then  
\begin{equation}
\label{Eq: Def_tau_k+1}
\tau_{k+1} := \inf \{n > \mu_k: L_{n+1}=0\}-\mu_k; 
\end{equation}
with the convention that  $\tau_{k+1}= \infty$, if the above set is empty. If $\mu_k =\infty$, then we set $\tau_{k+1} = 0$.

Similarly, if  $\mu_k < \infty$ and $\tau_{k+1}< \infty$, then
\begin{equation}
\label{Eq: Def_sigma_k+1}
\sigma_{k+1} := \inf \{n > \mu_k+\tau_{k+1}: L_{n+1}=1\} -\mu_k-\tau_{k+1},  
\end{equation}
with the convention that $\sigma_{k+1}= \infty$, if the above set is empty. If either $\mu_k= \infty$ or $\tau_{k+1}= \infty$, then we set $\sigma_{k+1} =0$.

In words, $\tau_k$ is the local time (that is, the time spent in) of state $\{1\}$ on the $k$-th visit to this state before its next exit from this state, and a similar explanation holds for $\sigma_k$. We will refer to these times $\tau_k$ and $\sigma_k$ as the inter-arrival times for the process $\{L_n\}_{n \ge 1}$.

As an immediate consequence of the above inter-arrival times, we have the following alternative representation for the collection $\DD$.
\begin{defi}[Equivalent definition for the nested sequences $\mathcal{D}$]
\label{Def: Alt_def_DD}Let $\mathcal{D}$ be a collection of nested sequence generated by the chain $\{L_n\}_{n \ge 1}$. Consider the sequences of random variables $\{ \tau_k, \sigma_k \}_{k \geq 1}$ where $\tau_k, \sigma_k$ as defined above. Then the sets $D_n^c$ and $D_n$ can be expressed as follows:  
for $k \in [\mu_n, \mu_{n+1})$, with $\mu_0 \equiv 0$
\[ D^c_k = \begin{cases}

    D^c_{\mu_n} \cup \{ \mu_n + 1, \mu_n + 2, \ldots, k \},  & \text{ for } \mu_n < k \leq \mu_n+\tau_{n+1}, \\ 
    D^c_{\mu_n+ \tau_{n+1}} & \text{ for } \mu_n + \tau_{n+1} < k \leq     \mu_{n+1},
\end{cases}\]
and 
\[
D_k = \begin{cases}
    D_{\mu_n} & \text{ for } \mu_n < k \le \mu_n+\tau_{n+1}, \\
    D_{\mu_n} \cup \{ \mu_n + \tau_{n+1}+1, \mu_n + \tau_{n+1}+2, \ldots, k \} & \text{ for } \mu_n < k \leq \mu_{n+1}.
\end{cases}
\]
\end{defi}
We make the following assumptions on the inter-arrival times $\{\tau_n\}_{n \ge 1}$ and $\{\sigma_n\}_{n \ge 1}$. 
\begin{enumerate}
\item[(\textbf{A1})] $\{ \tau_n \}_{n \geq 1}$ and$\{ \sigma_n \}_{n \geq 1}$ are i.i.d., and indenpendent of each other. Let $\tau$ denote a generic random variable such that $\tau \stackrel{d}{=} \tau_n$. Similarly, we define $\sigma$ to represent a generic random variable with the same law as $\sigma_n$.   
\item[(\textbf{A2})] There exists a $\delta>0$, such that, $\forall t \in (-\delta,\delta), $ $ \E[e^{t\tau}] < \infty$ and $\E[e^{t\sigma}] < \infty$.
\end{enumerate}

The i.i.d.\ assumption in (\textbf{A1}) on the sequence of inter-arrival times $\{\tau_n\}_{n \ge 1}$ and $\{\sigma_n\}_{n \ge 1}$ implies that the chain $\{L_n\}_{n \ge 1}$ has an underlying renewal structure as in regenerative chains (see \cite{Si_Wo_93} for a review on regenerative processes). The assumption (\textbf{A2}) implies that the process $\{L_n\}_{n \ge 1}$ is much more than positive recurrent, but in fact the return times have all moments finite, as we will see later. Our chain $L_n$ can have at most $2$ possible states; hence these assumptions are quite natural and satisfied by a large class of chains $L_n$, some of which we will discuss later.

\begin{defi}[Collection of nested sequences $\DD$ with renewal structure]
Let $\DD$ be a collection of nested sequences generated by a stochastic chain, such that it satisfies (\textbf{A1}) and (\textbf{A2}). Then we say that $\DD$ is a collection of nested sequences having an underlying renewal structure.
\end{defi}

\begin{example}
\label{Def: examples}
Here are some examples of nested sequences $\DD$ having an underlying renewal structure.  
\begin{itemize}
\item Let $\{L_n\}_{n \ge 1}$ be an i.i.d.\ sequence taking values in $\{0, 1\}$, such that $L_1 \equiv 1$, and let $D_n^c$ be constructed as before.

\item Consider a time-homogeneous irreducible Markov chain $\{L_n\}_{n \ge 1}$ taking values in $\{0, 1\}$, such that $L_1 \equiv 1$. 

\item Let $\tau_{1}\equiv k,$ and $\sigma_{1}\equiv \mu-k$, for some integer $\mu, k \in \Z_{+}$. In this case, we say that the collection $\DD$ is generated according to an arithmetic progression. 


\end{itemize}
\end{example}

The remainder of this paper is structured as follows. The rest of Section \ref{Sec:Intro} contains the main results and discussion on the open problems. Section \ref{Sec: techniques} introduces the restricted walks that are used to prove the main results, while Sections \ref{Sec: Proof_D_n_trivial} and \ref{sec: Proofs_D_n_non-trivial} contain their proofs. Section \ref{sec: Examples} elaborates further on Example \ref{Def: examples}. Appendices A, B, and C contain the supplementary material. 

\subsection{Main results}
\label{SubSec: Main_results}
Our main results are the law of large numbers and the central limit theorems for TMERW. In preparation for these, let us define the size of the sets $D_n$ and $D_n^c$ as 
\begin{equation}
\label{Eq: size_D}
|D_n|: = \sum_{k=1}^n \bone_{\{k \in D_n\}}, \text{ and }
|D^c_n|: = \sum_{k=1}^n \bone_{\{k \in D^c_n\}}.
\end{equation}
We will say that the size of $D_n$ is trivial if 
\begin{equation}
\label{Eq: size_D_n_0}
\frac{|D_n|}{n} \longrightarrow 0, \text{ a.s.}, 
\end{equation}
otherwise, we say that the size of $D_n$ is non-trivial.
\subsubsection{ Size of $D_n$ is trivial}
\label{Sec: D_n_size_trivial}
The following are the main results when the size of $D_n$ is trivial.
\begin{theorem}
\label{Thm: SLLN_non-nested} 
Let $S_n$ be a TMERW. 
Let $\mathcal{D}$ be any collection of sets that is not necessarily nested, such that $\frac{\mid D_n \mid }{n} \rightarrow 0, \text{a.s.}$, as $n \to \infty$. Then for $0 \le p<1$,
\begin{equation}
\label{Eq: SLLN_Dn_goes_to_zero}
\frac{S_n}{n}\longrightarrow 0, \text{ a.s.}
\end{equation}  
\end{theorem}


Next, we present the fluctuation results for the TMERW.
We observe a phase transition, similar to the classical ERW model, irrespective of whether the set $D_n$ corresponds to distant or recent memory.
\begin{theorem}
\label{Thm: CLT_D_n/n_goes_to_zero} 
Let $S_n$ be a TMERW. 
Let $\mathcal{D}$ be a collection of sets that are not necessarily nested, such that $\frac{\mid D_n \mid }{n} \rightarrow 0, \text{a.s.}$, as $n \to \infty$. Then, we have the following phase transitions into diffusive, critical, and superdiffusive regimes given by 
\begin{enumerate}[(i)]
\item If $ p < 3/4$ and $\frac{|D_n|}{\sqrt{n}}\stackrel{L^2}{\longrightarrow} 0 $, then 
\begin{equation*}
\label{Eq: D_n_goes_to_zero_sub-diffussive}
\frac{S_n}{\sqrt{n}} \stackrel{d}{\longrightarrow} N \left(0, \frac{1}{3-4p}\right).
\end{equation*}
\item If $p =\frac{3}{4}$, and $|D_n| = \mathcal{O}(n^{\gamma}), \text{ a.s.}$ for $0 \le \gamma < 1/2$, 
\begin{equation*}
\frac{S_n}{\sqrt{n \log n }}\stackrel{d}{\longrightarrow} N (0,1).
\end{equation*}
\item For $\frac{3}{4}<p<1$, and $|D_n| = \mathcal{O }(n^{\gamma}), \text{ a.s.}$ for $0 \le \gamma < 2p-1$, then there exists a random variable $L$ such that
\begin{equation}
\label{Eq: CLT_D_n_to_zero_p_ge_3/4}
\frac{S_n}{n^{2p-1}}\rightarrow L, \text{ a.s.}
\end{equation}
\end{enumerate}
\end{theorem}

\subsubsection{Size of $D_n$ is non-trivial}
\label{Subsec: Case I}
The next theorem is the law of large numbers when $|D_n|$ is non-trivial, under the additional assumption of $\DD$ being a nested sequence. 
\begin{theorem}
\label{Thm: SLLN_renewal_chain}
Let $S_n$ be a TMERW. Assume that $\DD$ is a collection of nested sequences having an underlying renewal structure, that is,  
the assumptions (\textbf{A1}) and (\textbf{A2}) hold. Then, as $n \to \infty$,
    \begin{equation}
    \label{eq:SLLN}
       \frac{S_n}{n} \longrightarrow \frac{(2\lambda-1)\E[\sigma]}{\mu - (2p-1)\E[\tau]} \; \text{ a.s.},
    \end{equation}
where $\mu = \E[\tau]+ \E[\sigma]$. 
\end{theorem}
\begin{remark}
\label{Re: DD_non_trivial}
Observe that in the statement of the above theorem, the non-triviality of the limit of $\frac{|D_n|}{n}$ is implied by the assumptions (\textbf{A1}) and (\textbf{A2}), which will be proved in Lemma \ref{Lem: SLLN for D}. 
\end{remark}
\begin{remark}
\label{Re: size_D_n_role}
Observe that in Theorem \ref{Thm: SLLN_non-nested}, the collection of sets $\DD$ is such that, $\frac{|D_n|}{n} \longrightarrow 0, \text{ a.s.}$, as $n \to \infty$, that is, the size of the set on which we tamper with the memory is trivial. In this case, we prove the SLLN for any collection $\DD$, nested or otherwise. However, in the case when $|D_n|$ is non-trivial, we require the additional assumption that $\DD$ is nested. We believe that the SLLN is still valid when $|D_n|$ is non-trivial, even when $\DD$ is non-nested. However, our proof does not carry forward in this case.
\end{remark}
An immediate corollary of Theorem \ref{Thm: SLLN_renewal_chain} is the following.
\begin{corollary}
\label{Cor: recurrence_tansience}
Let $S_n$ be a TMERW. Assume that $\DD$ is a collection of nested sequence having an underlying renewal structure, that is,  
the assumptions (\textbf{A1}) and (\textbf{A2}) hold. Then, as $n \to \infty$
\[
S_n \longrightarrow \begin{cases}
+\infty, & \text{if } \lambda >\frac{1}{2}, \\
-\infty, & \text{if } \lambda <\frac{1}{2}.
\end{cases}
\]
This shows that $S_n$ is transient when $\lambda \neq \frac{1}{2}$.
\end{corollary}
In this case, the fluctuation results that we have are only when the set $\DD$ corresponds to an arithmetic progression, as discussed in Examples \ref{Def: examples}. We believe that this is a technical difficulty; the fluctuation result can be proven for more general $\DD$ (see Problem \ref{Conj: CLT_general_tau_sigma}).
\begin{theorem}
\label{Thm: anamolous_diffusion_Dn_non_trivial}
Let $S_n$ be a TMERW. Let the collection of sets $\DD$ be a nested sequence generated according to an arithmetic progression, i.e for some fixed $\mu, k \in \bN$, $\PP \left(\tau =k\right)=\PP \left(\sigma =\mu-k\right)=1$. 
Write $ s^*=\frac{(2\lambda-1)(\mu-k)}{\mu - (2p-1)k}$. Let $\Sigma$ and $\widetilde{\Sigma}$ be two positive constants that depend on the parameters. Then,
\begin{enumerate}[i.]
\item If $k> \frac{\mu}{2}$, then there exists a phase transition into a diffusive, critical and superdiffusive behaviour characterised below:
\begin{enumerate}[a.]
\item When $p<\frac{1}{2}+\frac{\mu}{4k}$, then 
\begin{equation}
\label{Eq: CLT_renewal_sub_diffusive}
\frac{S_n -n s^*}{\sqrt{n}}\stackrel{d}{\longrightarrow} N (0, \Sigma),
\end{equation}
\item When $p= \frac{1}{2}+\frac{\mu}{4k}$, then 
\begin{equation}
\label{Eq: CLT_renewal_diffusive}
\frac{S_n- n s^* }{\sqrt{n \log n}}\stackrel{d}{\longrightarrow} N (0,\widetilde \Sigma),
\end{equation}
\item When $\frac{1}{2}+\frac{\mu}{4k} <p \le 1$, then there exists a random variable $\xi$, such that, 
\begin{equation}
\label{Eq: lim_renewal_super_diffusive}
\frac{S_n- n s^* }{n^{1-\rho^*}} \longrightarrow \xi, \text{ a.s.},
\end{equation}
where $\rho^*= 1- (2p-1)\frac{k}{\mu}$. 
\end{enumerate}
\item For $k =\frac{\mu}{2}$, if  $0 \le p <1$, 
\begin{equation}
\label{Eq: CLT_sqrt_n_equal parts memmory}
\frac{S_n -n s^*}{\sqrt{n}}\stackrel{d}{\longrightarrow} N (0, \Sigma),
\end{equation}
And when $p=1$,
\begin{equation}
\label{Eq: CLT_sqrt_n_logn_equal parts memmory}
\frac{S_n -n s^*}{\sqrt{n \log n}}\stackrel{d}{\longrightarrow} N (0, \widetilde\Sigma),
\end{equation}
\item If $k < \frac{\mu}{2}$, and $0 \le p \le 1$, then 
\begin{equation}
\label{Eq: CLT_sqrt_n_only}
\frac{S_n -n s^*}{\sqrt{n}}\stackrel{d}{\longrightarrow} N (0, \Sigma).
\end{equation}
\end{enumerate}
\end{theorem}

\begin{corollary}
\label{Cor: Anomolous_diffusion}
Let $S_n$ be a TMERW. Let the collection of sets $\DD$ be a nested sequence generated according to an arithmetic progression. Then, the superdiffuisive regime is observed only when $\lim_{n \to \infty}\frac{\lvert D_n^c \rvert}{n}>\frac{1}{2}$.
\end{corollary}

\begin{proposition}\label{prop: limiting variance}
Let $\eta = \frac{(2p-1)ks^*+(2\lambda-1)(\mu-k)}{\mu}$. Let $\Sigma$ and $\widetilde{\Sigma}$ be as in Theorem \ref{Thm: anamolous_diffusion_Dn_non_trivial}. Then \[\Sigma = (1-\eta^2)\left(\frac{k+ \dfrac{(\mu-k)^3}{k^2}}{1-2(2p-1)\dfrac{k}{\mu}} + \frac{\frac{2(\mu-k)^2}{k}-\frac{2(\mu-k)^3}{k^2}}{\left(1-(2p-1)\dfrac{k}{\mu}\right)}+\frac{(\mu-k)^3}{k^2} - \frac{2(\mu-k)^2}{k} + \mu-k \right),\] and 
\[\widetilde{\Sigma} = k + \frac{(\mu-k)^3}{k^2}.\]
Furthermore, when $k = \frac{\mu}{2}$, these simplify into $\Sigma = \frac{\mu}{2-2p}$ and $\widetilde{\Sigma} = \mu. $
\end{proposition}

Theorem \ref{Thm: anamolous_diffusion_Dn_non_trivial} demonstrates that the phase transition into the superdiffusive regime is observed only if the walk has a chance of performing a classical ERW more often than the simple random walk, whereas otherwise there is only the diffusive regime (unless $p=1$). This suggests that by appropriately choosing the memory to sample from, we can make the walk behave closer to a classical ERW or a simple random walk. Furthermore, observe that since $\mu \ge k$, the critical $p_c=\frac{1}{2}+\frac{\mu}{4k} \ge \frac{3}{4}$, which implies the diffusive regime is wider for TMERW as opposed to the classical ERW, where we know the critical parameter is at $p_c= \frac{3}{4}$.  

\subsection{Overview and discussion}

As observed earlier, when the size of the set $D_n$ is non-trivial, Theorem \ref{Thm: anamolous_diffusion_Dn_non_trivial} indicates that there is some sort of \textit{competition} between the classical ERW and the simple random walk, i.e. when the size of $D_n$ is trivial, the walk is performing mostly an ERW. This implies that the growth rates of $D_n$ assumed in Theorem \ref{Thm: CLT_D_n/n_goes_to_zero} might not be necessary, and are rather technical upshots of stochastic approximations. Thus, from our simulations (see Figures \ref{fig:img2},\ref{fig:img1},\ref{fig:img4} and \ref{fig:img3}) we conjecture the following:  
\begin{problem}
\label{Conj: Fluctuations_D_n_trivial}
When the size of $D_n$ is trivial, that is, $\lim_{n \to \infty}\frac{|D_n|}{n}=0$, the TMERW behaves qualitatively like a classical ERW, that is, the TMERW demonstrates phase transition into diffusive $p< 3/4$, critial $p=3/4$ and superdiffusive $p>3/4$ regimes irrespective of the size of $D_n$.    
\end{problem}
An immediate open problem related to the fluctuations of the TMERW when the size of $D_n$ is non-trivial is the following.
\begin{problem}
\label{Conj: CLT_general_tau_sigma}
Let $\DD$ be a collection of nested sequences with an underlying renewal structure satisfying assumptions (\textbf{A1}) and (\textbf{A2}). For $\frac{\E [\tau]}{\mu} > \frac{1}{2}$, we expect a phase transition into a superdiffusive regime with criticality at $p_c= \frac{3}{4}+\frac{1}{4} \frac{\E \left[\sigma\right]}{\E \left[\tau \right]}$. 
On the other hand, if $\frac{\E [\tau]}{\mu} < \frac{1}{2}$, then the TMERW will have only a diffusive regime, that is, a scaling of $\OO(\sqrt{n})$ for $0 \le p <1$.  
\end{problem}
Observe that when the size of $D_n$ is trivial we did not require the collection $\DD$ to be nested. This raises an interesting question of whether the sets $D_n$ are nested when the the size of $D_n$ is at all necessary. This is posed below. 
\begin{problem}
\label{Conj: non-nested_D_n_non_trivial}
The phase transition is completely determined by whether $\lim_{n\to \infty}\frac{|D_n^c|}{n}$ is less than, equal to or greater than $\frac{1}{2}$, independently of how the sets are positioned. 
\end{problem}

The Corollary \ref{Cor: recurrence_tansience} shows that the TMERW is transient when $\lambda \neq \frac{1}{2}$, while nothing is known for recurrence and transience in the case when the size of $D_n$ is trivial. Thus, the following question remains open.
\begin{problem}
Is the TMERW recurrent, that is, is $\PP \left(S_n=0 \text{ i.o. }\right)=1$? Or does it display a phase transition between recurrence and transience? If there is a phase transition, then the critical parameter should be jointly determined by $\lim_{n}\frac{|D_n|}{|D_n^c|}$ and $p$. A similar question can be asked in dimensions $d \ge 2$. 
\end{problem}

The model can be naturally extended to the case where the innovation set $D_n$ is partitioned into finitely many subsets $D_n^{(1)},\ldots,D_n^{(m)}$, on each of which the innovations evolve as independent simple random walks with parameters $\lambda_1,\ldots,\lambda_m$. The almost sure limit of $\frac{S_n}{n}$ is then modified according to the parameters $\lambda_i$ and the asymptotic proportions of the corresponding subsets. However, the memory breakpoint for the persistence of the diffusive, critical and superdiffusive phase transition remains unchanged. One may also allow the retained memory $D_n^c$ to be partitioned into finitely many subsets, each governed by an elephant random walk with parameter $p_i$. The resulting stochastic approximation is immediate to derive and can be analysed by the same methods. Since neither extension requires any essentially new ideas, we do not pursue them here.

\section{Structure of the Increments and Restricted Walks}
\label{Sec: techniques}
A key feature of the classical elephant random walk is that the conditional expectation of the next increment depends linearly on the current position of the walker. This one-dimensional structure underlies many of the existing martingale and stochastic approximation arguments. However, in the tampered memory model studied in this paper, this structure is lost because the conditional distribution of the next increment depends on whether the sampled index lies in the tampered or untampered part of the memory. Consequently, the current position of the walker alone no longer determines the conditional drift, since it does not distinguish between the contributions from the tampered and untampered parts of the memory.

To overcome this difficulty, we decompose the walk according to the tampered and untampered portions of the memory. This leads naturally to a two-dimensional stochastic approximation scheme (see \cite{Borkar, Zh_2016} for details on the concepts of stochastic approximation). Let $S_n$ denote the TMERW as defined in \eqref{Def: TERW}. We define  

\begin{equation}
\label{Def: TMERW_Dc}
S_n^{D^c}: = \sum_{j \in D_n^c} X_j,
\end{equation}
to be the TMERW restricted to the set $D_n^c$.
Similarly, define TMERW restricted to the set $D_n$ as 
\begin{equation}
\label{Def: TMERW_D}
S_n^{D}: = \sum_{j \in D_n} X_j.
\end{equation}
By construction, 
\begin{equation}
\label{Eq: S_n_sum_of_restricted_walks}
S_n = S_n^{D^c}+ S_n^{D}.
\end{equation}
Recall that the law of the innovations $Y_n$ is given by  $\PP(Y_n=1)=1-\PP(Y_n=-1)=\lambda$ and $\mathcal{D}$ and $\{Y_{n+1}\}_{n \ge 1}$ are independent of each other. Let $\FF_n = \sigma(\{X_1,\cdots X_n, D_1, \cdots, D_n\})$. Then, it is easy to see that  
\begin{eqnarray}
\label{Eq: conditional_expectation}
      \E[X_{n+1} \ | \ \FF_n] &= &
    \E \left [ \left(pX_{I_{n}}- (1-p)X_{I_{n}}\right)\bone_{\{I_{n} \in D_n^c\}} \ | \ \FF_n \right]+\E [\lambda \bone_{\{I_{n} \in D_n\}}-(1-\lambda) \bone_{\{I_{n} \in D_n\}} | \FF_n] \nonumber \\  
    &=& \frac{(2p-1)}{n}\sum_{j \in D_n^c} X_j + (2\lambda-1) \frac{|D_n|}{n} \nonumber \\ 
    &=& (2p-1) \frac{S_n^{{\tiny \DD}^c}}{n} + (2\lambda-1) \frac{|D_n|}{n}.
\end{eqnarray}
Unlike \cite{coletti2017central}, the conditional expectation of the next increment is not a linear function of the current position $S_n$. Instead, it depends on the contribution of the walk arising from the untampered memory, $S_n^{{\tiny \DD}^c}$ and the size of the tampered set $D_n$. This motivates the study of normalized restricted walks $\dfrac{S_n^{{\tiny \DD}^c}}{n}$ and $\dfrac{S_n^{{\tiny \DD}}}{n}$. Our approach is to set up recursions for stochastic approximation (SA) for the vector $\left(\frac{S_n^{{\tiny \DD}^c}}{n}, \frac{S_n^{{\tiny \DD}}}{n}\right)^\top$.

\section{Proof of the main results when the size of $D_n$ is trivial}
\label{Sec: Proof_D_n_trivial}

Recall from \eqref{Eq: size_D_n_0}, that we say that the size of $D_n$ is trivial, 
when $\frac{|D_n|}{n} \rightarrow 0$ a.s. In this case, we set up and analyse the recursion for stochastic approximation for $\frac{S_n}{n}$. Using \eqref{Eq: conditional_expectation}, and adding and subtracting $\E\left[X_{n+1} | \FF_n \right]$, where $\FF_n = \sigma(\{X_1,\cdots X_n, D_1, \cdots, D_n\})$, we get
\begin{eqnarray*} 
\frac{S_{n+1}}{n+1} &=& \frac{S_n}{n} + \frac{1}{n+1} \left [ X_{n+1} - \frac{S_n}{n} \right] \\
&=& \frac{S_n}{n} + \frac{1}{n+1} \left [(2p-1) \frac{S_n^{{\tiny \DD}^c}}{n} + (2\lambda-1) \frac{|D_n|}{n} - \frac{S_n}{n}\right] + \frac{1}{n+1} \Delta M_{n+1}\\
&=& \frac{S_n}{n} + \frac{1}{n+1} \left [(2p-2) \frac{S_n}{n} + (2\lambda-1) \frac{|D_n|}{n} - (2p-1)\frac{S_n^{D}}{n}\right] + \frac{1}{n+1} \Delta M_{n+1},
\end{eqnarray*}
where $\Delta M_{n+1} = X_{n+1} - \E[ X_{n+1} | \FF_n]$ is a martingale difference.
Let us rewrite the above recursion as in \eqref{Eq: recursion_general_SA}
\begin{equation}
\label{Eq: SA_Dn_to_zero}
\frac{S_{n+1}}{n+1}=\frac{S_n}{n} + \frac{1}{n+1} \left [h\left(\frac{S_n}{n} \right)+  \Delta M_{n+1} +r_{n+1}\right],
\end{equation}
where $h:\Rbold \rightarrow \Rbold$ is defined as $h(\theta)= 2(p-1)\theta$ is the mean-field function, and the error or the remainder term is given by 
\begin{equation}
\label{Eq: error_D_n/n_goes_to_zero}
r_{n+1}= (2\lambda-1) \frac{|D_n|}{n} - (2p-1)\frac{S_n^{D}}{n}.
\end{equation} 

\begin{proof}[Proof of Theorem \ref{Thm: SLLN_non-nested}]
Observe that $h(x)=2(p-1)x$ is a Lipschitz function.
Furthermore, $h^{'}(\theta)= 2(p-1)<0$, for $0 \le p <1$, which implies that Corollary \ref{Cor: SA_SLLN} is satisfied. Also, observe that $\lvert \frac{S_n^{D}}{n}\rvert \le \frac{\lvert D_n \rvert}{n}$, which shows that as $n \to \infty,$
\begin{equation}
\label{Eq: error_D_n_converge_to_zero}
\lvert r_{n+1}\rvert = \mathcal{O}\left(\frac{\lvert D_n \rvert}{n}\right)\rightarrow 0, \text{ a.s.}
\end{equation}
Observe that 
\begin{equation}
\label{Eq: martingale_difference_Dn_n_cgs_zero}
\Delta M_{n+1} = X_{n+1} - \E[ X_{n+1} | \FF_n]= X_{n+1}-(2p-1) \frac{S_n^{{\tiny \DD}^c}}{n} + (2\lambda-1) \frac{|D_n|}{n}
\end{equation}
Since $\frac{\lvert D_n \rvert}{n} \le 1$, and $\frac{\lvert S_n^{{\tiny \DD}^c}\rvert }{n} \le 1$, it is easy to see that 
\begin{equation}\label{Eq: boundedness of martingale trivial case}
    \sup_{n \ge 1} \E[| \Delta M_n |^2 \vert \FF_{n-1}] \le (1+\lvert (2p-1)\rvert + \lvert (2\lambda-1)\rvert)^2 < \infty, \text{ a.s.}
\end{equation} 
Therefore, from Theorem \ref{Thm: SA_SLLN}, we know that 
$\frac{S_n}{n} \rightarrow x^*$, where $x^*$ solves the equation $h(x)=0$. We observe that the only solution for $h(x)=0$, is $x^{*}=0$. This shows that as $n \to \infty$
\[
\frac{S_n}{n} \rightarrow 0, \text{ a.s.}
\]
\end{proof}
Next, we present the proof of the fluctuation results. 
\begin{proof}[Proof of Theorem \ref{Thm: CLT_D_n/n_goes_to_zero}]
Recall that the mean field function $h$ as in \eqref{Eq: SA_Dn_to_zero} is given by 
\[
h(\theta)= 2 (p-1) \theta, 
\]
and hence 
\begin{equation}
\label{Eq: eigenvalue_D_n/n_goes_to_zero}
-h'(\theta)= 2(1-p)
\end{equation}
Furthermore, recall that from \eqref{Eq: error_D_n_converge_to_zero} the error is given by 
\begin{equation}
\label{Eq: rate_of_error_D_n_goes_to_zero}
r_{n+1}= (2\lambda-1) \frac{|D_n|}{n} - (2p-1)\frac{S_n^{D}}{n} = \OO\left(\frac{\lvert D_n\rvert}{n}\right),
\end{equation}
i.e. there exists a uniform constant $C>0$, such that $|r_{n+1}| \le C \frac{\lvert D_n\rvert}{n}, \text{ a.s.}$ and the martingale difference is given by
\begin{equation}
\label{Eq: def_martingale_difference_D_n/n_converges_zero}
\Delta M_{n+1} = X_{n+1} - \E[ X_{n+1} | \FF_n]= X_{n+1}-(2p-1) \frac{S_n^{{\tiny \DD}^c}}{n} + (2\lambda-1) \frac{|D_n|}{n}.
\end{equation}

\textbf{(i)} To prove (i), observe that from \eqref{Eq: eigenvalue_D_n/n_goes_to_zero} $\rho^*=- h'(\theta)= 2(1-p)> \frac{1}{2} $, if $p < 3/4$. Therefore, we have to use Theorem \ref{Thm: CLT_SA_eigenvalue_more_than_half}. Furthermore, using \eqref{Eq: rate_of_error_D_n_goes_to_zero} and 
$\frac{|D_n|}{\sqrt{n}}\stackrel{L^2}{\longrightarrow} 0 $, we have
\begin{equation*}
\label{Eq:rate_of_error_D_n_goes_to_zero_subdiffusive}
(n+1) \E \left[r_{n+1}^2\right]\le C \E \left[\frac{\lvert D_n \rvert^2}{n}\right]\longrightarrow 0, \text{ as } n \to \infty. 
\end{equation*}

Therefore, by Theorem \ref{Thm: CLT_SA_eigenvalue_more_than_half}, we know that in this case 
\begin{equation*}
\frac{S_n}{\sqrt{n}} \stackrel{d}{\longrightarrow }N (0, \sigma^2),
\end{equation*}
where $\sigma^2$ is given by \eqref{Eq: variance_covariance_matrix_eigenvalue_more_than_half}, where $\Gamma = \lim_{n \to \infty}\E \left[ (\Delta M_{n+1})^2 | \ \FF_{n}\right].$
Observe that from \eqref{Eq: SLLN_Dn_goes_to_zero} and the fact that $\lvert \frac{S_n^D}{n}\rvert \leq \frac{\lvert D_n \rvert }{n}\rightarrow 0$, we know that as $n \to \infty$
\[
\frac{S_n^{{\tiny \DD}^c}}{n}= \frac{S_n}{n}- \frac{S_n^D}{n} \rightarrow 0 \text{ a.s.}
\]
Hence, using \eqref{Eq: conditional_expectation} and the definition of $\Delta M_n$ from \eqref{Eq: def_martingale_difference_D_n/n_converges_zero},we get as $n \to \infty$
\begin{equation}
\label{Eq: conditional_expectation_martingale_difference_limit}
\E \left[ (\Delta M_{n+1})^2 | \ \FF_{n}\right]= 1-\left((2p-1) \frac{S_n^{{\tiny \DD}^c}}{n} + (2\lambda-1) \frac{|D_n|}{n}\right)^2 \rightarrow 1, \text{ a.s.} 
\end{equation} 
Now using \eqref{Eq: variance_covariance_matrix_eigenvalue_more_than_half}, and the fact that $\Gamma=1$, we obtain
\[
\sigma^2= \int_{0}^{\infty}{\left(e^{-2(2(1-p)-\frac{1}{2})u}\right)\mathrm{d}u}= \frac{1}{2(\frac{3}{2}-2p)}=\frac{1}{3-4p}.
\]
This completes the proof of (i).

\noindent{\textbf{(ii)}} To prove (ii), we first observe that in this case $p=\frac{3}{4}$ implies $\rho^*=-h'(\theta)= 2(1-p)= \frac{1}{2}$. Hence, we need to verify the conditions in Theorem \ref{Thm: CLT_SA_for_rho_equals_1/2}. Recall that in our case the mean field function is given by $h(\theta)= 2(p-1)\theta$, which shows that \eqref{Eq: Taylor_expnasion_for_theta} is satisfied. To verify the corresponding version of the Lindeberg-Feller condition given in \eqref{Eq: Decay_martingale_difference_rho_half}, observe that the martingale difference $\Delta M_n$ is given by \eqref{Eq: def_martingale_difference_D_n/n_converges_zero}. So 
\begin{equation}
\label{Eq: bound_on_martingale_difference}
\sup_{n \ge 1} \lvert \Delta M_n \rvert \le 1 +\lvert 2p-1 \rvert + \lvert 2 \lambda -1 \rvert, \text{ a.s.}
\end{equation}
Hence, for any given $\epsilon>0$, $\sup_{m \ge 1} \E\left[( \Delta M_m )^2 \mathbf{1}_{\{(\Delta M_m)^2 \ge \epsilon \sqrt{n}\}}\mid  \FF_{m-1}\right] \rightarrow 0, \text{ a.s.}$ as $n \to \infty$. This shows that as $n \to \infty$
\[
\frac{1}{n} \sum_{m=1}^n \E\left[(\Delta M_m) ^2 \mathbf{1}_{\{\lvert \lvert \Delta M_m\rvert \rvert \ge \epsilon \sqrt{n}\}}\mid  \FF_{m-1}\right]\rightarrow 0, \text{ a.s.}
\]
To verify \eqref{Eq: sum_of_rn}, observe that using \eqref{Eq: rate_of_error_D_n_goes_to_zero}
\[
\sum_{k=1}^n r_k= \OO (\sum_{k=1}^n\frac{\lvert D_k \rvert }{k})=\OO (\sum_{k=1}^n k^{\gamma-1})=\OO(n^{\gamma})=o\left(\sqrt{\frac{n}{\log n}}\right), \text{ a.s.}, 
\]
since $ \lvert D_n \rvert = \mathcal{O}(n^{\gamma})$ a.s. where $\gamma< \frac{1}{2}$. This shows that using Theorem \ref{Thm: CLT_SA_for_rho_equals_1/2} 
\[
\frac{S_n}{\sqrt{n \log n }}\stackrel{d}{\longrightarrow} N (0,\widetilde{\sigma}^2),
\]
where $\widetilde{\sigma}^2$ is given \eqref{Eq: Description_tilde_Sigma}.
To find $\Gamma$ as in \eqref{Eq: Gamma_as_limit}, we obtain using \eqref{Eq: conditional_expectation_martingale_difference_limit},
\[
\frac{1}{n} \sum_{m=1}^n \E \left[\Delta M_m \left(\Delta M_m\right)^\top \mid \FF_{m-1}\right]\rightarrow 1 , \text { a.s.}, 
\]
which shows that in this case $\Gamma=1$. Therefore, to find $\widetilde{\sigma}^2$ we use \eqref{Eq: Description_tilde_Sigma} and $p=\frac{3}{4}$,  to show that 
\[
\widetilde{\sigma}^2= \lim_{n \to \infty}\frac{1}{\log n}\int_{0}^{\log n} { \mathrm{d}u} =1.
\]
This completes the proof of (ii).

\noindent(iii) In this case, we observe that $-h'(\theta)= 2(1-p) < \frac{1}{2}$, since we assumed $p> 3/4$. Therefore, we have to verify the conditions of Theorem \ref{Thm: CLT_SA_for_rho_leq_1/2}. To do so, firstly observe that  \eqref{Eq: bound_on_martingale_difference} shows that for a suitable constant $C>0$, 
$\sum_{m=1}^n \E \left[(\Delta M_m)^2 \mid \FF_{m-1}\right]  \le C n , \text{ a.s.}$ which verifies 
\eqref{Eq:_second_moment_martingale_difference_rho_leq_half}. 
To verify 
\eqref{Eq: error_rate_rho_leq_half}, observe that since $1- \rho^*= 2p-1$, and $\gamma< 2p-1$, there exists a $\delta_0>0$, such that, $ \delta_0 < 2p-1-\gamma$. For this choice of $\delta_0>0$, it is easy to see using  \eqref{Eq: rate_of_error_D_n_goes_to_zero}
\begin{equation}
\sum_{k=1}^n r_k= \OO (\sum_{k=1}^n\frac{\lvert D_k \rvert }{k})=\OO (\sum_{k=1}^n k^{\gamma-1})=\OO(n^{\gamma}) = o(n^{1-\rho^*-\delta_0}).
\end{equation}

This shows that there exists an $L$ such that \eqref{Eq: CLT_D_n_to_zero_p_ge_3/4} holds. This completes the proof for (iii).
\end{proof}

\section{Proofs of the main results when the size of $ D_n$ is non-trivial} \label{sec: Proofs_D_n_non-trivial}
To present the proofs of the main results, we require some preliminary observations for the collection $\DD$. These we present next. 
\subsection{Auxiliary results for $\DD$}
\label{SubSubsec: Auxiliary results for the collection D}
Throughout this section, $\DD$ will be a collection of nested sequences with renewal structure, that is, the Assumptions (\textbf{A1}) and (\textbf{A2}) hold.
\begin{lemma}
\label{Lem: tau_sigma_positive}
Let $\tau_n$ and $\sigma_n$ as defined in equations \eqref{Eq: Def_tau_1}- \eqref{Eq: Def_sigma_k+1}. Then, for all $n \ge 1$,
\begin{equation}
\label{Eq: tau_greater_than_1}
\PP \left(\tau_n \ge 1\right)=1,
\end{equation}
and 
\begin{equation}
\label{Eq: sigma_greater_than_1}
\PP \left(\sigma_n \ge 1\right)=1.
\end{equation}
Furthermore, 
\begin{equation}
\label{Eq: mu_greater_than_n}
\PP \left(\mu_n \ge n\right)=1.
\end{equation}
\end{lemma}
\begin{proof}
Observe that by definition we have $\PP\left(\tau_1 \ge  1 \right)=1$.
From the basic definitions of inter-arrival times \eqref{Eq: Def_tau_k+1}, we know that for $k >1$, 
\begin{eqnarray*}
\PP\left(\tau_k= 0\right) 
&= & \PP\left(\mu_k =\infty \right)\\
& = & \sum_{m=1}^{k-1} \PP\left(\mu_m =\infty, \mu_{m-1}< \infty\right).
\end{eqnarray*}
Now observe that 
\[
\PP\left(\mu_m =\infty, \mu_{m-1}< \infty\right)= \PP \left(\tau_{m} = \infty, \mu_{m-1}< \infty\right)+\PP \left(\sigma_{m} = \infty, \mu_{m-1}< \infty\right).
\]
It is easy that from the assumption of i.i.d.\ in (\textbf{A1}), we have 
\begin{equation*}
\PP \left(\tau_{m} = \infty, \mu_{m-1}< \infty\right)
= \PP \left(\tau= \infty\right) \PP \left(\mu_{m-1}< \infty\right).
\end{equation*}
Now we know from assumption (\textbf{A2}), that $\E \left[\tau \right] <\infty$, hence we get using Markov inequality 
\begin{equation*}
\PP \left(\tau= \infty\right)=\lim_{k \to \infty}\PP \left(\tau \ge k \right)\le \lim_{k \to \infty} \frac{1}{k} \E \left[\tau \right] =0.
\end{equation*}
This shows that
\[
\PP \left(\tau_{m} = \infty, \mu_{m-1}< \infty\right)=0.
\]
Similarly, it can be shown that 
\[
\PP \left(\sigma_{m} = \infty, \mu_{m-1}< \infty\right)=0.
\]
This shows that for any $k \ge 1$
\[
\PP\left(\tau_k= 0 \right) =0,
\]
which implies that 
\[
\PP\left(\tau_k \ge 1\right)=1
\]
Similarly, we can show that 
\[
\PP\left(\sigma_k \ge 1\right) =1.
\]
Since, $\mu_n = \sum_{k=1}^n (\tau_k+\sigma_k)$, it follows that $\PP \left(\mu_n \ge n \right)=1$.
\end{proof}
Observe that the sequence of inter-arrival times $\{\tau_n\}_{n \ge 1}$ and $\{\sigma_n\}_{n \ge 1}$ are i.i.d. from assumption (\textbf{A1}). It is easy to see from (\textbf{A2}), that $\E\left[\tau\right] < \infty$, and $\E \left[\sigma\right]<\infty$, which follows from the moment problem (see pg. 296 Section 30 \cite{Bil_2012}). Therefore, it follows that as $n \to \infty$,
\begin{equation}
\label{Eq: SLLN_tau}
\frac{1}{n}\sum_{k=1}^n \tau_k \longrightarrow \E\left[\tau\right], \text{ a.s.},  
\end{equation}
\begin{equation}
\label{Eq: SLLN_sigma}
\frac{1}{n}\sum_{k=1}^n \sigma_k \longrightarrow \E\left[\sigma\right], \text{ a.s.},  
\end{equation}
and 
\begin{equation}
\label{Eq: SLLN_mu}
\frac{\mu_n}{n} \longrightarrow \E\left[\tau\right]+\E\left[\sigma\right], \text{ a.s.}, 
\end{equation}
where we recall that $\mu_n= \mu_{n-1}+\tau_n+\sigma_n$.
\begin{lemma}
\label{Lem: SLLN for D}
Let $\{\tau_n\}_{n \ge 1}$ and $\{\sigma_n\}_{n \ge 1}$ denote the sequence of inter-arrival times for the collection $\DD$, that satisfy (\textbf{A1}) and (\textbf{A2}). Then, as $n \to \infty,$
\begin{equation}
\label{Eq: SLLN_D_c}
\frac{|D^c_{\mu_n}|}{n}
\longrightarrow \E\left[\tau \right] \text{ a.s.}, 
\end{equation}
\begin{equation}
\label{Eq: SLLN_D}
 \frac{|D_{\mu_n}|}{n}\longrightarrow \E\left[\sigma \right] \text{ a.s.}
\end{equation}
\end{lemma}
\begin{proof}
Observe that by \eqref{Eq: size_D}, we have $|D^c_{\mu_n}|= \sum_{k=1}^{\mu_n} \bone_{\{k \in D^c_{\mu_n}\}}$ and 
$|D_{\mu_n}| = \sum_{k=1}^{\mu_n} \bone_{\{k \in D_{\mu_n}\}}$. Therefore, it follows that 
\[
|D^c_{\mu_n}|= \sum_{k = 1}^n \tau_k,
\]
and 
\[
|D_{\mu_n}|= \sum_{k = 1}^n \sigma_k.
\]
The rest of the proof now follows from equations \eqref{Eq: SLLN_tau} and \eqref{Eq: SLLN_sigma}.
\end{proof}

\subsection{Proof of Theorem \ref{Thm: SLLN_renewal_chain}}
Our main strategy in the proof of Theorem \ref{Thm: SLLN_renewal_chain} is to prove a similar result for the TMERW $S_{\mu_n}$, which is done in Lemma \ref{Lem: SLLN_for_S_mu_n}, and then we use the renewal structure of the collection $\DD$ to prove Lemma \ref{lem:concludinglemma} to complete the proof. 

In preparation for the proof, in the next few steps, we set the recursion for the stochastic approximation which we will require in the proof of Lemma \ref{Lem: SLLN_for_S_mu_n}. To this end,   
recall that $S_n^{D^c}: = \sum_{j \in D_n^c} X_j$ denotes the TMERW restricted to the sets $D^c_n$ and $S_n^{D}: = \sum_{j \in D_n} X_j$ denotes the corresponding TMERW restricted to $D_n$ as in equations \eqref{Def: TMERW_Dc} and \eqref{Def: TMERW_D}.
We will set up the recursion for the stochastic approximation for both $S_{\mu_n}^{D^c}$, and $S_{\mu_n}^{D}$. Let us write 
\begin{equation*}
\label{Eq: recursion_for_S_on_Dc_renewal}
\frac{S^{D^c}_{\mu_{n+1}}}{n+1}= \frac{S^{D^c}_{\mu_{n}}}{n}+\frac{1}{n+1} \left(\sum_{j=1}^{\tau_{n+1}} X_{\mu_n +j}-\frac{S^{D^c}_{\mu_{n}}}{n}\right).
\end{equation*}
Writing $\GG_n = \sigma (\{X_1, X_2, \ldots, X_{\mu_n}; \tau_1, \tau_2, \ldots \tau_n, \sigma_1, \sigma_2, \ldots, \sigma_n\})$, and adding and subtracting the conditional expectation, we can re-write the above equation as 
\begin{equation}
\label{Eq: recursion_Dc_step1}
\frac{S^{D^c}_{\mu_{n+1}}}{n+1}= \frac{S^{D^c}_{\mu_{n}}}{n}+\frac{1}{n+1} \left[\E\left[\sum_{j=1}^{\tau_{n+1}} X_{\mu_n +j} \Big | \GG_{n}\right]-\frac{S^{D^c}_{\mu_{n}}}{n}+ \Delta M_{\mu_{n+1}}^{D^c} \right].
\end{equation}
where the martingale difference is given by 
\begin{equation}
\label{Eq: martingale_diff_D_c}
\Delta M_{\mu_{n+1}}^{D^c}:= \sum_{j=1}^{\tau_{n+1}} X_{\mu_n +j}- \E \left[\sum_{j=1}^{\tau_{n+1}} X_{\mu_n +j} \Big | \GG_n\right].
\end{equation}
Similarly, we write the following recursion for $S^{D}_{\mu_{n+1}}$ by 
\begin{equation}
\label{Eq: recursion_D_step1}
\frac{S^{D}_{\mu_{n+1}}}{n+1}= \frac{S^{D}_{\mu_{n}}}{n}+\frac{1}{n+1} \left[\E\left[\sum_{j=\tau_{n+1}+1}^{\tau_{n+1}+\sigma_{n+1}} X_{\mu_n +j} \Big | \GG_{n}\right]-\frac{S^{D}_{\mu_{n}}}{n}+ \Delta M_{\mu_{n+1}}^{D} \right].
\end{equation}
where the martingale difference is given by 
\begin{equation}
\label{Eq: martingale_diff_D}
\Delta M_{\mu_{n+1}}^{D}:= \sum_{j=\tau_{n+1}+1}^{\tau_{n+1}+\sigma_{n+1}} X_{\mu_n +j}- \E \left[\sum_{j=\tau_{n+1}+1}^{\tau_{n+1}+\sigma_{n+1}} X_{\mu_n +j}\Big | \GG_n\right].
\end{equation}
\begin{lemma} 
\label{lem:alpha_beta}
Suppose (\textbf{A1}) and (\textbf{A2}). Then, there exist $\alpha_{\mu_n}^{D^c}, \beta_{\mu_n}^{D^c}, \alpha_{\mu_n}^{D}$ and $\beta_{\mu_n}^{D}$,  
such that, almost surely, 
\begin{equation}
\label{condE_D^c} 
\E\left[ \sum\limits_{i=1}^{\tau_{n+1}} X_{\mu_{n} + i} \Big \vert \GG_n \right] =  \E[\alpha_{\mu_n}^{D^c} \vert \GG_n] \frac{S_{\mu_n}^{D^c}}{n} + \E[\beta_{\mu_n}^{D^c} \vert \GG_n], 
\end{equation}
and,
\begin{equation} 
\label{condE_D} 
\E\left[ \sum\limits_{i=1}^{\sigma_{n+1}} X_{\mu_{n} + \tau_{n+1} + i} \Big \vert \GG_n \right] =  \E[\alpha_{\mu_n}^{D} \vert \GG_n] \frac{S_{\mu_n}^{\sDD^c}}{n} + \E[\beta_{\mu_n}^{D} \vert \GG_n], 
\end{equation}
where
\begin{equation}
\label{Eq: limit_of_alpha_beta_D_c}
\E[\alpha_{\mu_n}^{D^c} \vert \GG_{n}] = (2p-1)\frac{\E[\tau]n}{\mu_n}+ \OO\left(\frac{1}{n}\right),\, \E[\beta_{\mu_n}^{D^c} \vert \GG_{n}] = (2\lambda-1)\frac{\E[\tau]|D_{\mu_n}|}{\mu_n} + \OO\left(\frac{1}{n}\right)
\end{equation}
and 
\begin{equation}
\label{Eq: Eq: limit_of_alpha_beta_D}
\E[\alpha_{\mu_n}^{D} \vert \GG_{n}] = (2p-1)\frac{\E[\sigma]n}{\mu_n} + \OO\left(\frac{1}{n}\right)\, \text{ and } 
\E[\beta_{\mu_n}^{D} \vert \GG_{n}] = (2\lambda-1)\frac{\E[\sigma]|D_{\mu_n}|}{\mu_n} + 
\OO\left(\frac{1}{n}\right),
\end{equation}
where the constant in $\OO \left(\frac{1}{n}\right)$ is uniform and deterministic almost surely. 

Furthermore, almost surely, 
 \begin{equation} 
 \label{eq:alpha_beta_limits}
 \lim_{n \to \infty} \E[\alpha_{\mu_n}^{D^c} \vert \GG_{n}]=(2p-1)\frac{\E[\tau]}{\mu}, \; \lim_{n \to \infty}\E[\beta_{\mu_n}^{D^c} \vert \GG_{n}]=(2\lambda-1)\frac{\E \left[\sigma \right]\E[\tau]}{\mu},  
 \end{equation} 
 and 
 \begin{equation}
 \label{Eq: alpha_beta_D_lim}
 \lim_{n \to \infty}\E[\alpha_{\mu_n}^{D} \vert \GG_n]= (2p-1)\frac{\E[\sigma]}{\mu}, \;
 \lim_{n \to \infty} \E[\beta_{\mu_n}^{D} \vert \GG_{n}]= (2\lambda-1)\frac{\E^2[\sigma]}{\mu}.
\end{equation}

\end{lemma}
\begin{proof}
    For $1 \leq k \le  \tau_{n+1}$, let $B_{\mu_n+k} = (2p-1) \frac{S_{\mu_n}^{\sDD^c}}{\mu_n+k} + (2\lambda-1) \frac{|D_{\mu_n+k}|}{\mu_n+k} $. \\
    Recall that $\GG_n = \sigma (\{X_1, X_2, \ldots, X_{\mu_n}; \tau_1, \tau_2, \ldots \tau_n, \sigma_1, \sigma_2, \ldots, \sigma_n\})$.
    Observe that from \eqref{Eq: conditional_expectation} and using tower property given $\tau_{n+1}=k+1$,
    \begin{eqnarray} 
    \label{eq:R_1}
\E\left[ \sum\limits_{i=1}^{k+1} X_{\mu_{n} + i} \Big \vert \GG_n \right]  &=& \E\left[ \E \left[\sum\limits_{i=1}^{k+1} X_{\mu_{n} + i} \vert \GG_{n}, X_{\mu_n+1},\dots X_{\mu_n+k} \right] \Big \vert \GG_{n} \right]\nonumber \\
&=& \E \left[ \frac{(2p-1)}{\mu_n+k}(S_{\mu_n}^{\sDD^c} + \sum\limits_{i=1}^{k} X_{\mu_{n} + i}) + (2\lambda-1) \frac{|D_{\mu_n+k}|}{\mu_n+k} \Big \vert \GG_{n} \right] + \E\left[ \sum\limits_{i=1}^{k} X_{\mu_{n} + i}  \Big \vert  \GG_{n} \right] \nonumber\\
     &=& B_{\mu_n+k} + \E\left[\left(1+\frac{2p-1}{\mu_n+k} \right) \sum\limits_{i=1}^{k} X_{\mu_{n} + i}  \Big \vert  \GG_{n} \right] \nonumber \\
    &=&  \sum_{j=0}^{k} \prod_{l=j+1}^{k} \left(1+\frac{2p-1}{\mu_n+l} \right) B_{\mu_n+j}.
\end{eqnarray}
Therefore, from the above expression we get, 
\begin{equation}
\label{Eq: condi_exp_first_tau}
\E\left[ \sum\limits_{i=1}^{\tau_{n+1}} X_{\mu_{n} + i} \Big \vert \GG_n \right]= \E\left[ \sum_{j=0}^{\tau_{n+1}-1} \prod_{l=j+1}^{\tau_{n+1}-1} \left(1+\frac{2p-1}{\mu_n+l} \right) B_{\mu_n+j}\Big \vert \GG_n \right]
\end{equation}
Define 
\begin{equation}
\label{Eq: Def_alpha_D_c}
\alpha_{\mu_n}^{D^c}= (2p-1) \sum_{j=0}^{\tau_{n+1}-1} \left(\prod_{l=j+1}^{\tau_{n+1}-1}\left(1+\frac{2p-1}{\mu_n+l}\right)\right)\frac{n}{\mu_n +j},
\end{equation} 
and
\begin{equation}
\label{Eq: Def_beta_D_c}
\beta_{\mu_n}^{D^c} =(2 \lambda -1)\sum_{j=0}^{\tau_{n+1}-1} \left(\prod_{l=j+1}^{\tau_{n+1}-1}\left(1+\frac{2p-1}{\mu_n+l}\right)\right)\frac{|D_{\mu_n +j|}}{\mu_n +j}.
\end{equation}
Therefore, from \eqref{Eq: condi_exp_first_tau}, we get 
\[
\E\left[ \sum\limits_{i=1}^{\tau_{n+1}} X_{\mu_{n} + i} \Big \vert \GG_n \right]= \E[\alpha_{\mu_n}^{D^c} \vert \GG_n]\frac{S_{\mu_n}^{D^c}}{n} + \E[\beta_{\mu_n}^{D^c} \vert \GG_n],
\]
which shows \eqref{condE_D^c}. 
Let us first evaluate $ \E[\alpha_{\mu_n}^{D^c} \vert \GG_n]$. This term can be simplified as follows. 
\begin{eqnarray*}
   \alpha_{\mu_n}^{D^c}= \sum_{j=0}^{\tau_{n+1}-1} \frac{(2p-1)n}{\mu_n+j}\left(1+\sum_{l=j+1}^{\tau_{n+1}-1} \left(\frac{2p-1}{\mu_n+l} \right) +\sum_{\substack{l_1, l_2 = j+1 \\ l_1\neq l_2}}^{\tau_{n+1}-1} \left(\frac{2p-1}{\mu_n+l_1} \right)\left(\frac{2p-1}{\mu_n+l_2} \right)+\cdots\right).
\end{eqnarray*}
$\alpha_{\mu_n}^{D^c}$ can be divided into two parts. Thus, the leading term can be approximated as 
\begin{eqnarray}
\label{Eq: term_lim_alpha_n}
\E\left[\sum_{j=0}^{\tau_{n+1}-1} \frac{(2p-1)n}{\mu_n+j} \Big\vert \GG_n\right] &=& (2p-1)\frac{n}{\mu_n} \E\left[\sum_{j=0}^{\tau_{n+1}-1} \frac{\mu_n}{\mu_n+j} \Big\vert \GG_n\right] \nonumber\\
&=&  (2p-1)\frac{n}{\mu_n}\E\left[\sum_{j=0}^{\tau_{n+1}-1} 1-\frac{j}{\mu_n+j} \Big\vert \GG_n\right].
\end{eqnarray}
Now observe that since $\{\tau_n\}_{n \ge 1}$ is a sequence of i.i.d.\, random variables, we
\begin{equation}
\label{Eq: sum_tau_summands}
\E\left[\sum_{j=0}^{\tau_{n+1}-1} \frac{j}{\mu_n+j} \Big\vert \GG_n\right] \le \E\left[\sum_{j=0}^{\tau_{n+1}-1} \frac{\tau_{n+1}}{\mu_n} \Big\vert \GG_n\right]= \frac{\E \left[\tau^2\right]}{\mu_n}.
\end{equation}
Therefore, from equations \eqref{Eq: term_lim_alpha_n} and \eqref{Eq: sum_tau_summands}, and using $\mu_n \ge n, \text{ a.s.}$ \eqref{Eq: mu_greater_than_n}, we get for an appropriate $\tilde{A}_n^{D^c}$
\begin{eqnarray} 
\label{Eq: Bound_term_first}
\E\left[\sum_{j=0}^{\tau_{n+1}-1} \frac{(2p-1)n}{\mu_n+j} \Big\vert \GG_n\right]
   &=& (2p-1)\frac{\E[\tau]n}{\mu_n} + \tilde{A}_n^{D^c},
\end{eqnarray}
where using the fact that $\mu_n \ge n, \text{ a.s.}$ from Lemma \ref{Lem: tau_sigma_positive}, we get 
\begin{equation}
\label{Eq: bound_on_A_tilde_c}
|\tilde{A}_n^{D^c}|\le \frac{n \E \left[\tau^2\right]}{\mu^2_n} \le \frac{\E \left[\tau^2\right]}{\mu_n}, \text{ a.s.}
\end{equation}
Next, we evaluate the second part. Without loss of generality, we may assume $2p-1>0$, otherwise we can work with $|1-2p|$. Then we get the following bound.
\begin{eqnarray*}
    \sum_{l=0}^{\tau_{n+1}-1} \left(\frac{2p-1}{\mu_n+l} \right) +\sum_{\substack{l_1, l_2 = 0 \\ l_1 \neq l_2}}^{\tau_{n+1}-1} \left(\frac{2p-1}{\mu_n+l_1} \right)\left(\frac{2p-1}{\mu_n+l_2} \right)+\cdots
    &\leq&  \sum_{l=0}^{\tau_{n+1}-1} \frac{2p-1}{\mu_n} +\sum_{ l_1\neq l_2}^{\tau_{n+1}-1} \frac{(2p-1)^2}{\mu_n^2} +\cdots \\
    &\leq& \sum_{k=1}^\infty \frac{(2p-1)^k\tau_{n+1}^k}{k!\mu_n^k}.  
\end{eqnarray*}
Thus, we get 
\begin{equation}
    \sum_{j=0}^{\tau_{n+1}-1} \frac{(2p-1)n}{\mu_n+j}\left(\sum_{l=j+1}^{\tau_{n+1}-1} \left(\frac{2p-1}{\mu_n+l} \right) +\sum_{\substack{l_1, l_2 = j+1 \\ l_1\neq l_2}}^{\tau_{n+1}-1} \left(\frac{2p-1}{\mu_n+l_1} \right)\left(\frac{2p-1}{\mu_n+l_2} \right)+\cdots\right) \leq \sum_{k=1}^\infty \frac{(2p-1)^{k+1}n\tau_{n+1}^{k+1}}{k!\mu_n^{k+1}}.
\end{equation}
Let $\delta>0$ be as in assumption (\textbf{A2}), i.e.  $\forall t\in (-\delta,\delta),$ $ \E[e^{t\tau}] < \infty$. Choose $0 <\delta_0 < \delta$. 
From \eqref{Eq: mu_greater_than_n}, we know that $\mu_n \ge n$, $\exists n_0$ such that $\PP \left(\sup_{n \ge n_0}\frac{2p-1}{\mu_n} \leq \delta_0 \right) =1$. Futhermore, from \eqref{Eq: mu_greater_than_n}, it follows that $\frac{\tau_n}{\mu_n} <1$, for large n. Also, from (\textbf{A2}) we know that all moments of $\E \left[\tau_n^k\right] < \infty$, for all $k \ge 1$, which follows from the moment problem (see pg. 296 Section 30 \cite{Bil_2012}). Therefore, 
\begin{eqnarray}
\label{Eq: Bound_term_second}
    \E\left[\sum_{k=1}^\infty \frac{(2p-1)^{k+1}n\tau_{n+1}^{k+1}}{k!\mu_n^{k+1}} \Big\vert \GG_n\right] &\leq& \frac{(2p-1)n}{\mu_n^2} \E\left[\sum_{k=1}^\infty \frac{(2p-1)^{k}\tau_{n+1}^{k}}{k!\mu_n^{k-2}} \Big\vert \GG_n\right] \nonumber \\
    &\leq& \frac{(2p-1)^3n}{\mu_n^2} \E\left[\sum_{k=1}^\infty \frac{\delta_0^{k-2}\tau_{n+1}^{k}}{k!} \Big\vert \GG_n\right] \nonumber\\
    &\leq& \frac{(2p-1)^3n}{\mu_n^2}\frac{1}{\delta_0^2} \E[e^{\delta_0\tau}] \le  \frac{C}{\mu_n}, 
\end{eqnarray}
for an appropriate constant $C>0$.
Thus, since $\mu_n \ge n \text{ a.s.}$, from equations \eqref{Eq: Bound_term_first}, \eqref{Eq: bound_on_A_tilde_c} and \eqref{Eq: Bound_term_second}, it follows that $\E[\alpha_{\mu_n}^{D^c} \vert \GG_n]=(2p-1)\frac{\E[\tau]n}{\mu_n} + \mathcal{O}\left(\frac{1}{n}\right), \text{ a.s.}$, and hence $\lim_{n \to \infty}\E[\alpha_{\mu_n}^{D^c} \vert \GG_n] = (2p-1)\frac{\E[\tau]}{\mu}$ a.s., since $\frac{\mu_n}{n} \stackrel{a.s.}{\longrightarrow} \mu$, as $n \to \infty$ which completes the proof of the first part of \eqref{Eq: limit_of_alpha_beta_D_c}.
Observe that from Definition \ref{Def: Alt_def_DD}, we know that for $0 \leq j \leq \tau_{n+1}$, $|D_{\mu_n+j}| = |D_{\mu_n}|$. Thus, from \eqref{Eq: SLLN_D}, \eqref{Eq: Def_alpha_D_c} and \eqref{Eq: Def_beta_D_c} and the first part of \eqref{Eq: limit_of_alpha_beta_D_c} which we have already established, we get as $n \to \infty$
\begin{eqnarray}
\label{Eq: Beta_D_c_calc}
     \E[\beta_{\mu_n}^{D^c} \vert \GG_n]
     &=& \frac{2\lambda-1}{2p-1}\frac{|D_{\mu_n}|}{n} \E[\alpha_{\mu_n}^{D^c} \vert \GG_n]\nonumber \\
     &=& (2\lambda -1) \frac{\E[\tau]|D_{\mu_n}|}{\mu_n} + \mathcal{O}\left(\frac{1}{n}\right) \stackrel{a.s.}{\longrightarrow} (2\lambda -1) \frac{\E \left[\sigma \right] \E[\tau]}{\mu},
\end{eqnarray}
Similar calculations as in  \eqref{eq:R_1} show that,
\begin{eqnarray*}
   & & \E\left[ \sum\limits_{i=1}^{\sigma_{n+1}} X_{\mu_{n} + \tau_{n+1} + i} \Big \vert \GG_n \right] \\ 
     &=& \E \left[\frac{2p-1}{\mu_{n+1}-1}\left(S_{\mu_n}^{\sDD^c} + \sum\limits_{j=1}^{\tau_{n+1}} X_{\mu_{n} + j}\right)+ \frac{2\lambda-1}{\mu_{n+1}-1}|D_{\mu_{n+1}-1}| \Big \vert \GG_n \right] 
     + \E\left[ \sum\limits_{i=1}^{\sigma_{n+1}-1} X_{\mu_{n} + \tau_{n+1} + i} \Big \vert  \GG_{n} \right] \nonumber \\
    \nonumber &=& \E \left[ \sum_{i=1}^{\sigma_{n+1}} \frac{2p-1}{\mu_{n+1}-i}\left(S_{\mu_n}^{\sDD^c} +  \sum\limits_{j=1}^{\tau_{n+1}} X_{\mu_{n} + j}  \right) + \frac{2\lambda-1}{\mu_{n+1}-i}|D_{\mu_{n+1}-i}| \Big \vert \GG_{n} \right]\\
    \nonumber &=& \E \left[ \sum_{i=1}^{\sigma_{n+1}} \frac{2p-1}{\mu_{n+1}-i}\left((n+ \alpha_{\mu_n}^{D^c})  \frac{S_{\mu_n}^{\sDD^c}}{n} + \beta_{\mu_n}^{D^c}   \right) + \frac{2\lambda-1}{\mu_{n+1}-i}|D_{\mu_{n+1}-i}| \Big \vert \GG_{n} \right]\\
    &=& \E \left[ \sum_{i=1}^{\sigma_{n+1}} \frac{(2p-1)n}{\mu_{n+1}-i} \Big \vert \GG_{n} \right]  \frac{S_{\mu_n}^{\sDD^c}}{n}  + \E \left[ \sum_{i=1}^{\sigma_{n+1}} \frac{2\lambda-1}{\mu_{n+1}-i}|D_{\mu_{n+1}-i}| \Big \vert \GG_{n} \right] + \mathcal{O}\left(\frac{1}{\mu_n}\right),
\end{eqnarray*}
since from equations \eqref{Eq: term_lim_alpha_n}, \eqref{Eq: Bound_term_second} and \eqref{Eq: Beta_D_c_calc}, it follows that \[\E \left[ \sum_{i=1}^{\sigma_{n+1}} \frac{2p-1}{\mu_{n+1}-i}\left( \alpha_{\mu_n}^{D^c}+ \beta_{\mu_n}^{D^c}\right)\Big \vert \GG_{n} \right] \leq \frac{(2p-1)\E[\sigma]}{\mu_n}\E \left[ \alpha_{\mu_n}^{D^c}+ \beta_{\mu_n}^{D^c}\Big \vert \GG_{n} \right]= \OO \left(\frac{1}{\mu_n}\right), \]
where the constant in $\OO \left(\frac{1}{\mu_n}\right)$ is uniform and deterministic a.s.
Writing 
\[
\alpha_{\mu_n}^{D}= \sum_{i=1}^{\sigma_{n+1}} \frac{(2p-1)n}{\mu_{n+1}-i}, 
\]
and using bounds similar to \eqref{Eq: Bound_term_first}, shows that as $n \to \infty$, 
\begin{equation}
     \E[\alpha_{\mu_n}^{D} \vert \GG_n] = (2p-1)\frac{\E[\sigma]n}{\mu_n} + \mathcal{O}\left(\frac{1}{\mu_n}\right) \stackrel{a.s.}{\longrightarrow} (2p-1)\frac{\E[\sigma]}{\mu}. 
\end{equation}
Define 
\[
\beta_{\mu_n}^{D} := \sum_{i=1}^{\sigma_{n+1}} \frac{2\lambda-1}{\mu_{n+1}-i}|D_{\mu_{n+1}-i}|.
\]
Recall that $\mu_{n+1}=\mu_{n}+\tau_{n+1}+\sigma_{n+1}$, and from Definition \ref{Def: Alt_def_DD}, we get  $|D_{\mu_n+\tau_{n+1}+j}| = |D_{\mu_n}|+j$ for $1 \le j \le \sigma_{n+1}$. Therefore, we get as $n \to \infty$, 
\begin{eqnarray}
     \E[\beta_{\mu_n}^{D} \vert \GG_n] &=&  \E \left[ \sum_{j=1}^{\sigma_{n+1}} \frac{(2\lambda-1)n}{\mu_{n}+\tau_{n+1}+j}\frac{|D_{\mu_{n}}|+j}{n} \Big \vert \GG_{n} \right] \nonumber\\
      &=&  \E \left[ \sum_{j=1}^{\sigma_{n+1}} \frac{(2\lambda-1)|D_{\mu_n}|} {\mu_{n}+\tau_{n+1}+j} \Big \vert \GG_{n} \right]  + \E \left[ \sum_{j=1}^{\sigma_{n+1}} \frac{(2\lambda-1)j}{\mu_{n}+\tau_{n+1}+j}\Big \vert \GG_{n} \right] \nonumber\\
     &=& \left((2\lambda-1)\frac{\E[\sigma]|D_{\mu_{n}}|}{\mu_n} + \mathcal{O}\left(\frac{1}{\mu_n}\right)\right) + \left((2\lambda-1)\frac{\E[\sigma]}{\mu_n} +\mathcal{O}\left(\frac{1}{\mu_n^2}\right)\right)\nonumber\\
     &=& (2\lambda-1)\frac{\E[\sigma]|D_{\mu_{n}}|}{\mu_n} + \mathcal{O}\left(\frac{1}{\mu_n}\right)\stackrel{a.s.}{\longrightarrow}(2\lambda-1)\frac{\E^2[\sigma]}{\mu}.
\end{eqnarray}
\end{proof}
Recall \eqref{Eq: recursion_Dc_step1}, and using \eqref{condE_D^c}, we get 
\begin{equation*}
\label{Eq: recursion_Dc_step2}
\frac{S^{D^c}_{\mu_{n+1}}}{n+1}= \frac{S^{D^c}_{\mu_{n}}}{n}+\frac{1}{n+1} \left[\E[\alpha_{\mu_n}^{D^c} \vert \GG_n] \frac{S_{\mu_n}^{\sDD^c}}{n} + \E[\beta_{\mu_n}^{D^c} \vert \GG_n]-\frac{S^{D^c}_{\mu_{n}}}{n}+ \Delta M_{\mu_{n+1}}^{D^c} \right].
\end{equation*}

Consider the following recursion
\begin{equation}
\label{Eq: recursion_bdd_time_D_c_right_normalisation}
\frac{S_{\mu_{n+1}^{D^c}}}{n+1}= \frac{S_{\mu_{n}^{D^c}}}{n}+ \frac{1}{n+1} \left[(\alpha^{D^c}-1)\frac{S_{\mu_{n}^{D^c}}}{n}+\beta^{D^c}+ \Delta M_{\mu{n+1}}^{D^c}+r_{\mu_{n+1}}^{D^c}\right],
\end{equation} 
where we recall from \eqref{eq:alpha_beta_limits} that
\begin{equation*}
\label{Eq:alpha_D_c_bounded_times_right_normalisation}
\alpha^{D^c}=:\lim_{n \to \infty} \mathbb{E}\left[\alpha_{\mu_n}^{D^c}| \GG_{n}\right]= (2p-1)\frac{\E \left[\tau\right]}{\mu},
\end{equation*}
and  
\begin{equation*}
\label{Eq:beta_D_c_bounded_times_right_normalisation}
\beta^{D^c}:=\lim_{n \to \infty}\E[\beta_{\mu_n}^{D^c} \vert \GG_{n}]=(2 \lambda -1) \frac{\E [\sigma] \E[\tau]}{\mu}.
\end{equation*}
And the martingale difference $\Delta M_{\mu_{n+1}}$ is defined in \eqref{Eq: martingale_diff_D_c}. Furthermore, 
the error $r_{\mu_{n+1}}^{D^c}$ is defined as 
\begin{equation}
\label{Eq: error_D_c_mu_n}
r_{\mu_{n+1}}^{D^c}:= \left(\E \left[\alpha_{\mu_n}^{D^c} | \GG_{n} \right]-\alpha^{D^c}\right)\frac{S_{\mu_n}^{D^c}}{n}+ \E \left[\beta_{\mu_n}^{D^c}| \GG_{n}\right]- \beta^{{D^c}}.
\end{equation}
Similarly, from \eqref{condE_D}, the recursion in \eqref{Eq: recursion_D_step1} can be written as 
\begin{equation}
\label{Eq: recursion_bdd_time_D_right_normalisation}
\frac{S_{\mu_{n+1}}^{D}}{n+1}= \frac{S_{\mu_{n}}^{D}}{n}+ \frac{1}{n+1} \left[\alpha^{D}\frac{S_{\mu_{n}}^{D^c}}{n} - \frac{S_{\mu_{n}}^{D}}{n}+\beta^{D}+ \Delta M_{\mu_{n+1}}^{D}+r_{\mu_{n+1}}^{D}\right],
\end{equation}
where from \eqref{Eq: alpha_beta_D_lim},
\begin{equation*}
\label{Eq:alpha_D_bounded_times_right_normalisation}
\alpha^{D}=:\lim_{n \to \infty} \mathbb{E}\left[\alpha_{\mu_n}^{D}| \GG_{n}\right]=(2p-1)\frac{\mathbb{E}[\sigma]}{\mu},
\end{equation*}
and 
\begin{equation*}
\label{Eq:beta_D_bounded_times_right_normalisation}
\beta^{D} :=\lim_{n \to \infty}\E[\beta_{\mu_n}^{D^c} \vert \GG_{n}] =(2\lambda-1)\frac{\mathbb{E}^2[\sigma]}{\mu}.
\end{equation*}
And, the martingale difference is given by \eqref{Eq: martingale_diff_D}
and the error $r_{\mu_{n+1}}^{D}$ is defined as 
\begin{equation}
\label{Eq: error_D_mu_n}
r_{\mu_{n+1}}^{D}:= \left(\E\left[\alpha_{\mu_n}^{D} | \GG_{n} \right]-\alpha^{D}\right)\frac{S_{\mu_n}^{D^c}}{n}+ \E \left[\beta_{\mu_n}^{D}| \GG_{n}\right]- \beta^{{D}}.
\end{equation}
Writing $\bS_n:= \left(S_n^{D^c}, S_{n}^D \right)^\top$, we can write the vector form of the stochastic approximation as 
\begin{equation}
\label{Eq: recursion_vector_matrix}
\frac{\bS_{\mu_{n+1}}}{n+1} = \frac{\bS_{\mu_{n}}}{n}+ \frac{1}{n+1} \left[A \frac{\bS_{\mu_{n}}}{n}+ \boldsymbol{\beta}+ \Delta \bM_{\mu_{n+1}} + \br_{\mu_{n+1}}\right], 
\end{equation}
where 
\begin{equation}
\label{Eq: matrix_A}
A = \begin{pmatrix}
     (2p-1)\frac{\E [\tau]}{\mu}-1 & 0\\
     (2p-1)\frac{\E [\sigma]}{\mu} & -1
\end{pmatrix}
\end{equation}
and 
\begin{equation}
\boldsymbol{\beta}:= \left(\beta^{D^c}, \beta ^D\right)^\top,
\end{equation}
\begin{equation}
\label{Eq: martingale_diff_vector}
\Delta \bM_{ \mu_{n}}:= \left(\Delta M_{\mu_n}^{D^c}, \Delta M_{\mu_n}^{D}\right)^\top,
\end{equation}
and 
\begin{equation}
\label{Eq: error_vector}
\br_{ \mu_{n}}:= \left( r_{\mu_n}^{D^c}, r_{\mu_n}^{D}\right)^\top.
\end{equation}
\begin{lemma}
\label{Lem: eigenvalues_of_A}
Let $A$ be as in \eqref{Eq: matrix_A}. Then $-A$ has two distinct positive eigenvalues given by $\eta_1=1$, and $\eta_2= 1-(2p-1)\frac{\E [\tau]}{\mu}$. Furthermore, if $\rho^*:= \min\{\eta_1, \eta_2\}$, then $\rho^* =1$, iff $p \le \frac{1}{2}$, and $\rho^*= \eta_2 <1,$ iff $p > \frac{1}{2}$.
\end{lemma}
\begin{proof}
It is immediate that the eigenvalues of $-A$ are given $\eta_1=1$ and $\eta_2=1-(2p-1)\frac{\E [\tau]}{\mu}$. We show that $\eta_2>0$, for any $0 \le p \le 1$. For $0 \le p \le \frac{1}{2}$, so $(2p-1)\le 0$ and from \eqref{Eq: tau_greater_than_1} we know that $\E [\tau] \ge 1$, and $\mu \ge 1$.  Therefore, $\eta_2\ge 1$. For $\frac{1}{2}< p \le 1$, $0< (2p-1) \le 1,$ and since we have from \eqref{Eq: sigma_greater_than_1}$\mu = \E [\tau] + \E [\sigma]$ and $\E [\sigma] \ge 1$, we have $\frac{\E [\tau]}{\mu}< 1$, which shows that $\eta_2>0$. Then, we have $\rho^*:= \min\{1, 1- (2p-1)\frac{\E [\tau]}{\mu} \}$. Furthermore, observe that $\eta_2 =1$, only when $p =\frac{1}{2}$. Therefore, it follows that $\rho^* =1$, iff $p \le \frac{1}{2}$, and $\rho^*= \eta_2 <1,$ iff $p > \frac{1}{2}$.
\end{proof}
\begin{lemma}
\label{lem: martingale_difference_bounded}
Let $\Delta \bM_{ \mu_{n}}$ be as in \eqref{Eq: martingale_diff_vector}. Then, almost surely, 
$\sup_{n \ge 1} \E \left[\lVert \Delta \bM_{ \mu_{n}}\rVert ^{2+\delta} \vert \GG_n\right]<\infty$ for any $\delta>0$.
\end{lemma}
\begin{proof}
Recall from \eqref{Eq: martingale_diff_vector} that $\Delta \bM_{ \mu_{n}}=\left(\Delta M_{\mu_n}^{D^c}, \Delta M_{\mu_n}^{D}\right)^\top$, where $\Delta M_{\mu_{n+1}}^{D^c}$ is as defined in \eqref{Eq: martingale_diff_D_c} and $\Delta M_{\mu_{n+1}}^{D}$ is as defined in \eqref{Eq: martingale_diff_D}.  
We will first show that 
$\sup_{n \ge 1}\E \left[\lvert\Delta M_{\mu_{n+1}}^{D^c}\rvert ^{2+\delta}| \GG_n\right]<\infty.$ 
Recall from \eqref{Eq: martingale_diff_D_c} that 
$\Delta M_{\mu_{n+1}}^{D^c}= \sum_{j=1}^{\tau_{n+1}} X_{\mu_n +j}- \E \left[\sum_{j=1}^{\tau_{n+1}} X_{\mu_n +j} \Big | \GG_n\right]$. Therefore, since $|X_k| \le 1$, for all $k \ge 1$, and from assumptions (\textbf{A1}) and (\textbf{A2}) and from Jensen's inequality, we have 
\begin{eqnarray*}
\E \left[\lvert\Delta M_{\mu_{n+1}}^{D^c}\rvert ^{2+\delta}| \GG_n\right] & \le & 2^{1+\delta}\left(\E \left[\left(\sum_{j=1}^{\tau_{n+1}} X_{\mu_n +j}\right)^{2+\delta} \Big | \GG_n \right]+\E^{2+\delta} \left[\sum_{j=1}^{\tau_{n+1}} X_{\mu_n +j} \Big | \GG_n\right]\right)\\
& \le & 2^{2+\delta}\E \left[\left(\sum_{j=1}^{\tau_{n+1}} X_{\mu_n +j}\right)^{2+\delta} \Big | \GG_n \right]\\
& \le &  2^{2+\delta} \E \left[\tau_{n+1}^{2+\delta}\Big | \GG_n\right]
= 2^{2+\delta} \E \left[\tau^{2+\delta}\right]<\infty, 
\end{eqnarray*}
which shows that $\sup_{n \ge 1}\E \left[\lvert\Delta M_{\mu_{n+1}}^{D^c}\rvert ^{2+\delta}| \GG_n\right]< \infty$.
Similarly, it can be shown that $\sup_{n \ge 1}\E \left[\lvert\Delta M_{\mu_{n+1}}^{D}\rvert ^{2+\delta}| \GG_n\right]< \infty$. This completes the proof.
\end{proof}
\begin{lemma}
\label{Lem: SLLN_for_S_mu_n}
Let us assume (\textbf{A1}) and (\textbf{A2}). Then 
\begin{equation}
\label{Eq:SLLN_for_S_mu_n}
\frac{S_{\mu_{n}}}{n} \longrightarrow \frac{\mu (2\lambda-1)\E[\sigma]}{\mu - (2p-1)\E[\tau]} \; \text{ a.s.}
\end{equation}
\end{lemma}
\begin{proof}
It is easy to see that in \eqref{Eq: recursion_vector_matrix}, the mean function $h(\mathbf{x})= A \mathbf{x} +\boldsymbol{\beta}$ is Lipschitz. Also, we know from Lemma \ref{lem: martingale_difference_bounded}, the martingale differences
$\sup_{n \ge 1} \E \left[\lVert \Delta \bM_{ \mu_{n}}\rVert ^2 \vert \GG_n\right]<\infty$. The error $\br_{ \mu_{n}} \longrightarrow 0, \text{ a.s.}$ from Lemma \ref{lem:alpha_beta}. Hence, this satisfies all conditions i-iv in Theorem \ref{Thm: SA_SLLN}. Furthermore, since we know from Lemma \ref{Lem: eigenvalues_of_A} that the eigenvalues of $A$ are distinct and strictly negative which implies by Corollary \ref{Cor: SA_SLLN} is satisfied. 
Thus, from the recursion in \eqref{Eq: recursion_vector_matrix}, it follows that as $n \to \infty$, 
\begin{equation}
\label{Eq: SLLN_for_vector_S_mu_n}
\frac{\bS_{\mu_{n}}}{n} \longrightarrow \mu \bs^*, \text{ a.s.}
\end{equation}
where $\mu \bs^*$ is the unique solution to $A \mathbf{x} + \boldsymbol{\beta} =0$, and is given by $\mu \bs^*= \left(s_1^*, s_2^*\right)^\top$, where $s_1^*= \frac{ (2\lambda-1)\E\left[\tau \right]\E[\sigma]}{\mu - (2p-1)\E[\tau]} $ and $s_2^*= \frac{ (2\lambda-1)\E^2[\sigma]}{\mu - (2p-1)\E[\tau]}$
Now \eqref{Eq:SLLN_for_S_mu_n} follows from \eqref{Eq: SLLN_for_vector_S_mu_n} by observing that $S_{\mu_n}= \langle \bS_{\mu_{n}}, \mathbf{1} \rangle$, where $<\cdot, \cdot>$ denotes inner product, $\mathbf{1}= \left(1,1\right)^\top$, and $\mu = \E \left[\tau\right]+ \E \left[\sigma\right]$. 
\end{proof}
\begin{lemma}\label{lem:concludinglemma}
Suppose (\textbf{A1}) and (\textbf{A2}) hold. Then, $\frac{S_n}{n} \to \frac{(2\lambda-1)\E[\sigma]}{\mu - (2p-1)\E[\tau]}$ almost surely as $n \to \infty$.
\end{lemma}
\begin{proof}
Let $k(n):= \sup \{k : \mu_k \le n\}$ denote the number of combined updates of both $\{\tau_n\}_{n \ge 1}$ and $\{\sigma\}_{n \ge 1}$ upto time $n$.
From \eqref{Eq: SLLN_mu}, we know that 
\begin{equation*}
\frac{\mu_n}{n} \stackrel{a.s.}{\longrightarrow} \E \left[\tau\right]+ \E \left[\sigma\right] =\mu.
\end{equation*}

For all $n \geq 1$, there exists $k(n)$, such that, 
\[ 
\mu_{k(n)} \leq n < \mu_{k(n)+1},
\]
and since $\E \left[\mu_1\right]< \infty$, $k(n) \nearrow \infty$, we have  
$\frac{k(n)}{n} \to \frac{1}{\mu}, \text{a.s.}$, as $n \to \infty$. We have,
\begin{eqnarray*}
\frac{S_n}{n} &=& \frac{1}{n}S_{\mu_{k(n)}} + \frac{1}{n}\sum\limits_{i=1}^{n-\mu_{k(n)}} X_{\mu_{k(n)}+i} \\
&=& \frac{k(n)}{n}\frac{S_{\mu_{k(n)}}}{k(n)}+ \frac{1}{n}\sum\limits_{i=1}^{n-\mu_{k(n)}} X_{\mu_{k(n)}+i}.
\end{eqnarray*}
It is easy to see that
\begin{eqnarray*}
\Big \vert \frac{1}{n}\sum\limits_{i=1}^{n-\mu_{k(n)}} X_{\mu_{k(n)}+i} \Big \vert & \leq & 1- \frac{\mu_{k(n)}}{n} \\
&=& 1-\frac{\mu_{k(n)}}{k(n)}\frac{k(n)}{n} \xrightarrow{a.s.} 0.
\end{eqnarray*}
Hence, it follows that 
\[ 
\lim_{n \to \infty} \frac{S_n}{n} = \lim_{n \to \infty} \frac{S_{\mu_{k(n)}}}{k(n)}\frac{k(n)}{n} \xrightarrow{a.s.} s^*,
\]
where $s^*= \frac{ (2\lambda-1)\E[\sigma]}{\mu - (2p-1)\E[\tau]}.$ This completes the proof.
\end{proof}
\begin{proof}[Proof of Theorem \ref{Thm: SLLN_renewal_chain}]
Combining Lemmas \ref{Lem: SLLN_for_S_mu_n} and \ref{lem:concludinglemma} completes the proof.
\end{proof}
\begin{proof}[Proof of Corollary \ref{Cor: recurrence_tansience}]
Observe that it is enough to show that $\mu - (2p-1)\E[\tau]>0$. The proof is trivial when $p=0$ or $p=1$. Therefore, we only consider the case when $0 <p <1$.  
If $0 < p < \frac{1}{2}$, then $(2p-1)\E[\tau]\le 0$, and hence $\mu - (2p-1)\E[\tau]>0$, since $\mu \ge \E[\sigma]>0$.
For $\frac{1}{2}< p \le 1$, $(2p-1)\E[\tau]\ge 0$.
Since $\E [\tau] \le \mu$, we get 
$\mu - (2p-1)\E[\tau]\ge \mu- (2p-1)\mu =2 \mu(1-p)>0 $, since $\mu \ge \E[\sigma]>0$.

\end{proof}

\subsection{Proof of Theorem \ref{Thm: anamolous_diffusion_Dn_non_trivial}}
Observe that when the collection of sets $\DD$ are generated by an arithmetic progression, then $\mu_n =n \mu$, and for any $n = j \mu  +m$ for $1 \le m \le \mu $, we have 
\[
D^c_n= \begin{cases}
D^c_{j \mu } \cup \{j \mu +1, j \mu+2, \ldots, j \mu +m \} & \text{ if } 1 \le m \le k, \\
D^c_{j \mu +k}, & \text{ if } k < m \le \mu, 
\end{cases}
\]
and 
\[
D_n = \begin{cases}
D_{j \mu } & \text{ if } 1 \le m \le k, \\
D_{j \mu } \cup \{j \mu +k+1, j \mu + k +2, \ldots,j \mu  +m \} & \text{ if } k+1< m \le \mu.
\end{cases}
\]
All the lemmas in this subsection are for the case when $\DD$ is generated in this manner. Observe that in this case $\GG_n = \FF_{\mu n}$. 

\begin{lemma} \label{lem:crossterms}
For all $i \neq j \in \{0, 1, \dots, \mu-1 \}$
\[ \lim_{n \to \infty} \E\left[ X_{\mu n+i} X_{\mu n+j} \vert \FF_{\mu n} \right] = \eta^2, \]
where $\eta = \frac{1}{\mu}((2p-1){k}s^* +(2\lambda-1)(\mu-k))$. 
\end{lemma}
\begin{proof}
    Consider $i, j$ such that $1 \leq i < j \leq \mu$. We have,
    \[ \E[X_{\mu n+i}X_{\mu n+j}\vert \FF_{\mu n}] = \E \left[ X_{\mu n+i} \E[X_{\mu n+j}\vert  \FF_{\mu n + j-1} ] \Bigm| \FF_{\mu n} \right]. \]
    Since, \[\E[X_{\mu n+j}\vert \FF_{\mu n + j-1}] = (2p-1) \frac{S_{\mu n+j-1}^{D^c}}{\mu n+j-1} + (2\lambda-1) \frac{|D_{\mu n+j-1}|}{\mu n+j-1}.\]
    and, $S_{\mu n+j-1}^{D^c} = S_{\mu n}^{D^c} + \sum_{r=1}^{j-1} \I_{\{\mu n+r\in D^c\}} X_{\mu n+r}$. 
It follows that
\begin{align*}
E[X_{\mu n+j}\mid  \FF_{\mu n + j-1} ]
&=
(2p-1)
\frac{
S_{\mu n}^{D^c}
+
\sum_{r=1}^{j-1}
\I_{\{\mu n+r\in D^c\}}
X_{\mu n+r}
}
{\mu n+j-1}
\\
&\qquad
+
(2\lambda-1)
\frac{|D_{\mu n+j-1}|}{\mu n+j-1}.
\end{align*}
Substituting this into the tower property gives
\begin{align*}
\E[X_{\mu n+i}X_{\mu n+j}\mid \FF_{\mu n}]
&=
\frac{2p-1}{\mu n+j-1}
E\left[
X_{\mu n+i}
\left(
S_{\mu n}^{D^c}
+
\sum_{r=1}^{j-1}
\I_{\{\mu n+r\in D^c\}}
X_{\mu n+r}
\right)
\Bigm|
\FF_{\mu n}
\right]
\\
&\qquad
+
(2\lambda-1)
\frac{|D_{\mu n+j-1}|}{\mu n+j-1}
\E[X_{\mu n+i}\mid \FF_{\mu n}].
\end{align*}
Expanding the first expectation,
\begin{align*}
\E[X_{\mu n+i}X_{\mu n+j}\mid \FF_{\mu n}]
&= \frac{2p-1}{\mu n+j-1}
\sum_{r=1}^{j-1}
\I_{\{\mu n+r\in D^c\}}
\E[X_{\mu n+i}X_{\mu n+r}\mid \FF_{\mu n}]
\\
& + \frac{2p-1}{\mu n+j-1}
S_{\mu n}^{D^c}
E[X_{\mu n+i}\mid \FF_{\mu n}] +
(2\lambda-1)
\frac{|D_{\mu n+j-1}|}{\mu n+j-1}
\E[X_{\mu n+i}\mid \FF_{\mu n}].
\end{align*}

Since $j\le \mu$ is fixed, the sum in the first term contains at most $\mu$ terms. Since $\vert \E[X_{\mu n+i}X_{\mu n+r}\mid \FF_{\mu n}] \vert
\leq 1$, for some constant $C>0$
\begin{equation*} 
\left|
\frac{2p-1}{\mu n+j-1}
\sum_{r=1}^{j-1}
\I_{\{\mu n+r\in D^c\}}
\E[X_{\mu n+i}X_{\mu n+r}\mid \FF_{\mu n}]
\right|
\le
\frac{C}{n}    
\end{equation*}
Thus,
\begin{align*}
\E[X_{\mu n+i}X_{\mu n+j}\mid \FF_{\mu n}]
&=
\Bigg(
(2p-1)\frac{S_{\mu n}^{D^c}}{\mu n+j-1}
+ (2\lambda-1)\frac{|D_{\mu n+j-1}|}{\mu n+j-1}
\Bigg)
\E[X_{\mu n+i}\mid \FF_{\mu n}] +o(1).
\end{align*}
Since $i$ and $j$ are fixed,
\[
\frac{S_{\mu n}^{D^c}}{\mu n+j-1} = \frac{S_{\mu n}^{D^c}}{\mu n} +o(1)
\quad \text{ and } \quad 
\frac{|D_{\mu n+j-1}|}{\mu n+j-1} =
\frac{\mu-k}{\mu}+o(1).
\]
Similarly,
\[ 
\E[X_{\mu n+i}\mid \FF_{\mu n}] = (2p-1)\frac{S_{\mu n}^{D^c}}{\mu n} + (2\lambda-1)\frac{\mu-k}{\mu}
+o(1).
\]
Using $\frac{S_{\mu n}^{D^c}}{n} \longrightarrow ks^*$, 
we obtain $\E[X_{\mu n+i}\mid \FF_{\mu n}]
\longrightarrow \eta$,
where $\eta = \frac{(2p-1)ks^*+(2\lambda-1)(\mu-k)}{\mu}$. Similarly, 
\[
(2p-1)\frac{S_{\mu n}^{D^c}}{\mu n+j-1}
+ (2\lambda-1)\frac{|D_{\mu n+j-1}|}{\mu n+j-1} \longrightarrow \eta.
\]
Therefore,
\[
\E[X_{\mu n+i}X_{\mu n+j}\mid \FF_{\mu n}]
\longrightarrow
\eta^2,
\]
which completes the proof.
\end{proof}

\begin{lemma}
\label{lem:Gamma}
Let $ \Delta \bM_{ \mu{n}}= \left(\Delta M_{\mu n}^{D^c}, \Delta M_{\mu n}^{D}\right)^\top$,
be as in \eqref{Eq: martingale_diff_vector}. 
Define $\Gamma \coloneqq\lim_{n \to \infty}  \E[ \Delta {\mathbf{M}}_{\mu n + \mu}  \Delta {\mathbf{M}}^\top_{\mu n + \mu} \vert \FF_{\mu n}]$. Then,
\[ \Gamma = (1-\eta^2) \begin{pmatrix}
     k & 0\\
     0 & \mu - k
 \end{pmatrix}, \]
where $\eta$ is as defined in Lemma~\ref{lem:crossterms}. Furthermore, $\Gamma$ is positive semi-definite. 
\end{lemma}
\begin{proof}
Let ${\mathbf{S}}_{\mu n} = \begin{pmatrix} \frac{S^{D^c}_{\mu n}}{n}, \frac{S^{D}_{\mu n}}{n} \end{pmatrix}^\top$ and ${\mathbf{X}}_{\mu n} = \begin{pmatrix} \sum_{j=1}^{k} X_{\mu n+j}, \; \sum_{j=k+1}^{\mu} X_{\mu n + j} \end{pmatrix}^\top$. Then, 
we have
\[ \E[{\mathbf{X}}_{\mu n + \mu} \vert \FF_{\mu n}] = A_n {\mathbf{S}}_{\mu n} + {\bf \beta}_{\mu n}, \] 
where as $n \to \infty$ , $A_n \to A + I$ and ${\bf \beta_{\mu n}} \to (2 \lambda -1) \dfrac{\mu-k}{\mu} \begin{pmatrix} k \\ \mu-k \end{pmatrix}$. We have,
\begin{eqnarray*}
  \E[ \Delta {\mathbf{M}}_{\mu n + \mu}  \Delta {\mathbf{M}}^\top_{\mu n + \mu} \vert \FF_{\mu n}] &=& \E\left[ \left( {\mathbf{X}_{\mu n + \mu}} - \E[{\mathbf{X}_{\mu n + \mu}} \vert \FF_{\mu n}] \right)   \left( {\mathbf{X}_{\mu n + \mu}} - \E[{\mathbf{X}_{\mu n + \mu}} \vert \FF_{\mu n}] \right)^\top \vert \FF_{\mu n} \right]   \\
  &=& \E\left[ {\mathbf{X}_{\mu n + \mu}}{\mathbf{X}^\top_{\mu n + \mu}}  \vert \FF_{\mu n} \right] - \E[{\mathbf{X}_{\mu n + \mu}} \vert \FF_{\mu n}] \E[{\mathbf{X}_{\mu n + \mu}} \vert \FF_{\mu n}]^\top     \\
\end{eqnarray*}
So,
\begin{eqnarray} \label{eq:gamma}
    \Gamma &=& \lim_{n \to \infty}  \E[ \Delta {\mathbf{M}}_{\mu n + \mu}  \Delta {\mathbf{M}}^\top_{\mu n + \mu} \vert \FF_{\mu n}] \nonumber \\
    &=& \lim_{n \to \infty} \E\left[ {\mathbf{X}_{\mu n + \mu}}{\mathbf{X}^\top_{\mu n + \mu}} \vert \FF_{\mu n} \right] - \lim_{n \to \infty} \E[{\mathbf{X}_{\mu n + \mu}} \vert \FF_{\mu n}] \E[{\mathbf{X}_{\mu n + \mu}} \vert \FF_{\mu n}]^\top  \nonumber \\
      &=& \lim_{n \to \infty} \E\left[ {\mathbf{X}_{\mu n + \mu}}{\mathbf{X}^\top_{\mu n + \mu}}\vert \FF_{\mu n} \right] \nonumber \\
      && - \left((A+I) s^* \begin{pmatrix} k \\ \mu-k \end{pmatrix}  + (2 \lambda -1) \frac{\mu - k}{\mu} \begin{pmatrix} k \\ \mu-k \end{pmatrix} \right) \left((A+I) s^* \begin{pmatrix} k \\ \mu-k \end{pmatrix}  + (2 \lambda -1) \frac{\mu-k}{\mu} \begin{pmatrix} k \\ \mu-k \end{pmatrix} \right)^\top  \nonumber \\
       &=& \lim_{n \to \infty} \E\left[ {\mathbf{X}_{\mu n + \mu}}{\mathbf{X}^\top_{\mu n + \mu}} \vert \FF_{\mu n} \right] - B \begin{pmatrix} k^2 & k(\mu-k)\\ k(\mu-k) & (\mu-k)^2 \end{pmatrix} B^\top,
\end{eqnarray}
where $B=(A+I)s^*+(2\lambda-1)\dfrac{\mu - k}{\mu}I$. Recall that $A+I
= \frac{2p-1}{\mu} \begin{pmatrix}
k & 0\\
\mu-k & 0
\end{pmatrix}$. We obtain,
\begin{align} \label{eq:term2}
& B \begin{pmatrix} k^2 & k(\mu-k)\\ k(\mu-k) & (\mu-k)^2 \end{pmatrix} B^\top \nonumber \\
&= \left((A+I)s^*+(2\lambda-1)\frac{\mu-k}{\mu}I\right)
\begin{pmatrix}
k^2 & k(\mu-k) \nonumber\\
k(\mu-k) & (\mu-k)^2
\end{pmatrix}
\left((A+I)s^*+(2\lambda-1)\frac{\mu-k}{\mu}I\right)^\top \nonumber \\
&=
\eta^2
\begin{pmatrix}
k^2 & k(\mu-k)\\
k(\mu-k) & (\mu-k)^2
\end{pmatrix}.
\end{align}
Using Lemma~\ref{lem:crossterms}, we get
\begin{align} \label{eq:crossterms}
& \lim_{n \to \infty} \E\left[ {\mathbf{X}_{\mu n + \mu}}{\mathbf{X}^\top_{\mu n + \mu}} \vert \FF_{\mu n} \right] \nonumber \\
&=   \lim_{n \to \infty} \begin{pmatrix}
\E\left[\left(\sum\limits_{i=1}^{k}X_{\mu n+i}\right)^2\Bigm|\FF_{\mu n}\right]
&
\E\left[\left(\sum\limits_{i=1}^{k}X_{\mu n+i}\right)
\left(\sum\limits_{j=k+1}^{\mu}X_{\mu n+j}\right)
\Bigm|\FF_{\mu n}\right]
\\
\E\left[\left(\sum\limits_{i=1}^{k}X_{\mu n+i}\right)
\left(\sum\limits_{j=k+1}^{\mu}X_{\mu n+j}\right)
\Bigm|\FF_{\mu n}\right]
&
\E\left[\left(\sum\limits_{i=k+1}^{\mu}X_{\mu n+i}\right)^2
\Big|\FF_{\mu n}\right]
\end{pmatrix}  \nonumber \\
&=
\begin{pmatrix}
k + 2\sum_{1\le i<j\le k} \E[X_{\mu n+i}X_{\mu n+j}|\FF_{\mu n}]
&
\sum_{i=1}^{k}\sum_{j=k+1}^{\mu} \E[X_{\mu n+i}X_{\mu n+j}|\FF_{\mu n}]
\\
\sum_{i=1}^{k}\sum_{j=k+1}^{\mu} \E[X_{\mu n+i}X_{\mu n+j}|\FF_{\mu n}]
&
\mu-k + 2 \sum_{k+1\le i<j\le\mu} \E[X_{\mu n+i}X_{\mu n+j}|\FF_{\mu n}]
\end{pmatrix} \nonumber \\
& =  \begin{pmatrix}
    k + k(k-1) \eta^2 & k(\mu-k) \eta^2 \\ k(\mu-k) \eta^2 & \mu-k + (\mu-k)(\mu-k-1) \eta^2
\end{pmatrix}.
\end{align}
Substituting \eqref{eq:crossterms} and \eqref{eq:term2} into \eqref{eq:gamma} we get
\[ \Gamma = (1-\eta^2) \begin{pmatrix}
    k & 0 \\ 0 & \mu-k \end{pmatrix}. \]
By definition, $\Gamma$ is positive semi-definite. This can also be verified directly by showing that $|\eta | \leq 1$. Further for $0 < k < \mu$, the matrix $\Gamma$ is positive definite, except when $(p, \lambda) = (1, 1)$ or $(p, \lambda) = (1, 0)$, in which case $\Gamma$ is the zero matrix.
\end{proof}

\begin{lemma}
\label{Lemma: error_diffusive_case}
Let $\br_{ \mu{n}}$ be the error as in \eqref{Eq: error_vector}. Then, as $n \to \infty$, 
\begin{equation}
\label{Eq: rate_second_moment_error_r_n}
(n+1)\E \left[\lVert \br_{ \mu{n}} \rVert^2\right] \rightarrow 0,
\end{equation}
and 
\begin{equation}
\label{Eq: rate_of_sum_of_errors}
\sum_{k=1}^n \br_{ \mu{n}} = o\left(\sqrt\frac{n}{\log n}\right), \text{ a.s.} 
\end{equation} 
Furthermore, if $\rho^*<\frac{1}{2}$ then for any $0< \delta_0< \frac{1}{2}$, 
\begin{equation}
\label{Eq: rate_of_sum_of_errors_super_diffusive}
\sum_{k=1}^n \br_{ \mu{n}} = o\left(n^{1-\rho^* -\delta_0}\right), \text{ a.s.} 
\end{equation} 
\end{lemma}
\begin{proof}
 Recall from \eqref{Eq: error_vector}, we have $\br_{ \mu{n}}:= \left( r_{\mu n}^{D^c}, r_{\mu n}^{D}\right)^\top$, where 
\[
r_{\mu{n+\mu}}^{D^c}:= \left(\E \left[\alpha_{\mu{n+\mu}}^{D^c}| \FF_{\mu n}\right]-\alpha^{D^c}\right)\frac{S_{\mu n}^{D^c}}{n}+ \E \left[\beta_{\mu{n+\mu}}^{D^c}| \FF_{\mu n}\right]- \beta^{{D^c}}
\]
and 
\[
r_{\mu{n+\mu}}^{D}:= \left(\E \left[\alpha_{\mu{n+\mu}}^{D} | \FF_{\mu n} \right]-\alpha^{D}\right)\frac{S_{\mu n}^{D^c}}{n}+ \E \left[\beta_{\mu{n+\mu}}^{D} | \FF_{\mu n} \right]- \beta^{{D}}
\]
from equations \eqref{Eq: error_D_c_mu_n} and \eqref{Eq: error_D_mu_n}.
Since $\lvert \frac{S_{\mu n}^{D^c}}{n}\rvert \le 1,$ we have 
\[
\lvert r_{\mu{n+\mu}}^{D^c} \rvert ^2 \le 2 \left( \E \left[\alpha_{\mu{n+\mu}}^{D^c}| \FF_{\mu n}\right]-\alpha^{D^c}\right)^2+ 2\left(\E \left[\beta_{\mu{n+\mu}}^{D^c}| \FF_{\mu n}\right]- \beta^{D^c}\right)^2. 
\] 
From Lemma \ref{lem:alpha_beta}, we know that 
$\E[\alpha_{\mu n}^{D^c} \vert \FF_{\mu n}] = \alpha^{D^c} + \OO\left(\frac{1}{n}\right)$, $\E[\beta_{\mu n}^{D^c} \vert \FF_{\mu n}] = \beta^{D^c} + \OO\left(\frac{1}{n}\right)$. Therefore, by the dominated convergence theorem, we have for an appropriate constant $C>0$, as $n \to \infty$
\begin{equation}
\label{Eq: r_mu_n_D_c_cgs_moment_2}
(n+1)\E\left[ \left(r_{\mu{n+\mu}}^{D^c}\right)^2\right] \le  C \left(\frac{1}{n}\right) \longrightarrow 0. 
\end{equation}
We can similarly prove $ (n+1)\E\left[ \left(r_{\mu{n+\mu}}^{D}\right) ^2\right]\longrightarrow 0$ as $n \to \infty$, and hence this together with \eqref{Eq: r_mu_n_D_c_cgs_moment_2} proves \eqref{Eq: rate_second_moment_error_r_n}.
Since, $\br_{ \mu{n}} = \OO(\frac{1}{n}), \text{ a.s.}$, we have 
as $n \to \infty$
\begin{equation*}
\sum_{k=1}^n \br_{ \mu{k}} = \mathcal{O}\left(\log n\right), \text{ a.s.}
\end{equation*} which shows 
\begin{equation}
\label{Eq: rate_of_sum_error1}
\sqrt{\frac{\log n}{n}}\sum_{k=1}^n  \br_{ \mu{n}} \le C \sqrt{\frac{(\log n)^3}{n}}\rightarrow 0, \text{ a.s.},
\end{equation}
which proves \eqref{Eq: rate_of_sum_of_errors}. Furthermore, we also get as $n \to \infty$, 
\begin{equation*}
\label{Eq: rate_of_sum_error2}
\sum_{k=1}^n  \br_{ \mu{n}} =\OO \left( \frac{\log n}{n^{1-\rho^* -\delta_0}} \right) \longrightarrow 0, \text{ a.s.}
\end{equation*}
since $1-\rho^* - \delta_0>0$, as $\rho^*< \frac{1}{2}$, and $0< \delta_0 < \frac{1}{2}$.
\end{proof}

\begin{remark}
    Lemma \ref{lem:Gamma} can be proven for a general $\DD$ satisfying (\textbf{A1}) and (\textbf{A2}). However the rates in Lemma \ref{Lemma: error_diffusive_case} do not hold in a general setting, as in this case  $\E[\alpha_{\mu_n}^{D^c} \vert \FF_{\mu n}] - \alpha^{D^c} = (2p-1)\E[\tau]\left(\frac{n}{\mu_n} - \frac{1}{\mu}\right)  + \OO\left(\frac{1}{n}\right)$,  $\E[\beta_{\mu_n}^{D^c} \vert \FF_{\mu n}] - \beta^{D^c} = (2\lambda-1)\E[\tau]\left(\frac{|D_{\mu_n}|}{\mu_n} - \frac{\E[\sigma]}{\mu}\right) + \OO\left(\frac{1}{n}\right)$ and so on. Thus, the rate of convergence of $ \br_{ \mu_{n}}$ to zero depends on the rate of convergence of $\frac{\mu_n}{n}$ to $\mu$ in this case.
\end{remark}




\color{black}

\begin{proof}[Proof of Theorem \ref{Thm: anamolous_diffusion_Dn_non_trivial}]
Recall that $\frac{\bS_{\mu n}}{\mu n} \longrightarrow  \bs^*, \text{ a.s.}$. 
We start by calculating the threshold value for $p$ with respect to the minimum eigenvalue $\rho^*$.
As seen in Lemma \ref{Lem: eigenvalues_of_A}, when $p\leq \frac{1}{2}$, $\rho^* = 1 > \frac{1}{2}$. Thus, we focus on $p>\frac{1}{2}$. Here, $\rho^* = 1-(2p-1)\frac{k}{\mu} = \frac{1}{2}$ iff $(2p-1)\frac{k}{\mu} = \frac{1}{2}$. Thus, we have the regime change at $p= \frac{1}{2} + \frac{\mu}{4k}$. Observe that we have the three cases based on the value of $\frac{k}{\mu}$, since when $\frac{k}{\mu} < \frac{1}{2}$ all values of $p$ satisfy the condition for $\rho^* > \frac{1}{2}$.

Lemma \ref{lem: martingale_difference_bounded} shows that conditions \eqref{eq: SA CLT cond1 rho>1/2} and \eqref{Eq: Decay_martingale_difference_rho_half} are satisfied by the martingale differences $\Delta M_{\mu n}$.
Next, observe that Lemma \ref{lem:Gamma} shows that there exists a positive semi-definite matrix $\Gamma$ that satisfies \eqref{eq: SA CLT cond2 rho>1/2}, \eqref{Eq: Gamma_as_limit} and \eqref{Eq:_second_moment_martingale_difference_rho_leq_half}. Similary using Lemma \ref{Lemma: error_diffusive_case}, we can say that $r_{\mu n}$ satisfies \eqref{Eq: error_for_second_eigenvalue_leq_1/2}, \eqref{Eq: sum_of_rn} and \eqref{Eq: error_rate_rho_leq_half} respectively depending on the value of $\rho^*$. 
Thus, all conditions in Theorem \ref{Thm: CLT_SA_eigenvalue_more_than_half} are satisfied and we have a CLT with $\sqrt{n}$ scaling when $\rho^* > \frac{1}{2}$. 
Observe that $h(\theta) = A \theta + \beta$ satisfies conditions \eqref{Eq: Taylor_expnasion_for_theta} and \eqref{Eq: Taylor_expnasion_for_theta_1}.
Thus, all the conditions in Theorem \ref{Thm: CLT_SA_for_rho_equals_1/2} are satisfied which gives us the appropriate CLT for $\rho^* = \frac{1}{2}$. Similarly, since conditions in Theorem \ref{Thm: CLT_SA_for_rho_leq_1/2} are satisfied, $\frac{S_{\mu n}}{n^{1-\rho^*}}$ converges to a limiting random variable $\xi$ a.s. when $\rho^* < \frac{1}{2}$.
By the Cram\'{e}r-Wold device (see Theorem 15.55 \cite{klenke2008probability}), we get the required results for $S_{\mu n}$.

Now observe that given any $n$, there exists $k(n) \ge 0$ and $1 \le b(n) \le \mu$, such that, $n = k(n) \mu + b(n)$. Therefore, $S_{n}= S_{ \mu k(n)}+ \sum_{j=k(n) \mu +1}^{b(n)}X_j$. If $f(n)$ denotes the required scaling in the various regimes, then we know that $f(n)\longrightarrow \infty $ as $n \to \infty$. Moreover, $\frac{f(k(n)\mu)}{f(n)} \longrightarrow 1$, as $n \to \infty$.  Write 
\[
\frac{S_n - n s^*}{f(n)}= \frac{f(k(n)\mu)}{f(n)}\frac{S_{ \mu k(n)} - k(n)\mu s^*}{f(k(n)\mu)}+ \frac{\sum_{j=k(n) \mu +1}^{b(n)}X_j - b(n) s^*}{f(n)} 
\]
Since $1 \le b(n) \le \mu$, we get $\Big | \frac{\sum_{j=k(n) \mu +1}^{b(n)}X_j - b(n) s^*}{f(n)}\Big | \le \frac{C}{f(n)} \longrightarrow 0$ for an appropriate constant $C>0$. Furthermore, since $k(n) \longrightarrow \infty,$ as $n \to \infty$, the required result is obtained.
\end{proof}
\begin{proof}
[Proof of Corollary \ref{Cor: Anomolous_diffusion}]
 Observe that for any $n \in \bN$, there exists $k(n) \ge 0$, and $1 \le b(n) \le \mu$, such that, $n = k(n) \mu +b(n)$, and $k(n) \to \infty$, as $n \to \infty$. In this case, $\lvert D_n^c\rvert = k(n)k+ \min\{k, b(n)\}$, and hence,
\[
\frac{\lvert D_n^c\rvert}{n}= \frac{k(n)k}{n}+ \frac{\min\{k, b(n)\}}{n} \longrightarrow \frac{k}{\mu}, 
\]
since $1 \le \min\{k, b(n)\} \le \mu$.
The rest of the proof is an immediate consequence of Theorem \ref{Thm: anamolous_diffusion_Dn_non_trivial}.
\end{proof}
\begin{proof}[Proof of Proposition \ref{prop: limiting variance}] 
According to Theorem \ref{Thm: CLT_SA_eigenvalue_more_than_half}, the variance-covariance matrix when $\rho^*>\frac{1}{2}$ is given by $\int_{0}^{\infty}{\left(e^{(A+\frac{1}{2}I) u}\right)^\top \Gamma e^{(A+\frac{1}{2}I) u} \mathrm{d}u}$,
where $\Gamma$ is the matrix obtained in Lemma \ref{lem:Gamma}. Similarly, by Theorem \ref{Thm: CLT_SA_for_rho_equals_1/2}, when $\rho^*=\frac{1}{2}$ the variance-covariance matrix is given by $ \lim_{n \to \infty}\frac{1}{\log n}\int_{0}^{\log n} {\left(e^{A+\frac{1}{2}I) u}\right)^\top \Gamma e^{(A+\frac{1}{2}I) u} \mathrm{d}u}$. Our limiting variance is the sum of the individual entries of the variance-covariance matrix. Calculations of the above integrals can be found in Appendix B.
\end{proof}

\section{Further discussion on examples} \label{sec: Examples}
In this section, we discuss in detail the examples introduced in Subsection \ref{Def: examples} and explain how Assumptions (\textbf{A1}) and (\textbf{A2}) are satisfied. 
\begin{example}[Markov chain]
Consider a time-homogeneous irreducible Markov chain $\{L_n\}_{n \ge 1}$ taking values in $\{0, 1\}$, such that $L_1=1$ w.p. $1$. The set $D^c_n$ is constructed recursively as follows: Let $D^c_1 = \{1\}$, and 
\begin{equation}
\label{Eq: Def_D_n}
D^c_{n+1}= \begin{cases}
D^c_n \cup \{n+1\}, & \text{ if } L_{n+1}=1,\\
D^c_n, & \text{ otherwise.}
\end{cases}
\end{equation}
Let the transition probabilities be given by 
\begin{equation*}
\PP\left(L_{n+1}=i \Big \vert L_n=i\right)=\pi_i, 
\end{equation*}
for $i=0,1$. Since we assume that $\{L_n\}_{n \ge 1}$ is irreducible, thus $0< \pi_i <1$ for $i=0,1$. Since $\{L_n\}_{ n\ge 1}$ is Markovian, hence the inter-arrival times are i.i.d.\, and independent of each other, satisfying assumption (\textbf{A1}).

Furthermore, since $\{L_n\}_{n \ge 1}$ is a recurrent chain, which implies that for all $n \ge 1$, we have $\PP \left(\tau_n < \infty\right)=1$ and $\PP \left(\sigma_n < \infty\right)=1$. Observe that by definition of the inter-arrival time, we have $L_{\tau_1+1}=0$, and $L_{\mu_n+\tau_{n+1}+1}=0$.And similarly, $L_{\mu_n+1}=1$ for all $n \ge 0$.  Observe that for $k \ge 1$
\[
\PP \left(\tau_1 =k \right)= \PP \left(L_{k+1}=0 \Big \vert L_1=1\right)= \pi_1^{k-1} (1-\pi_1).
\]
Similarly, since $\PP \left(\tau_1 < \infty \right)=1$ and by Strong Markov property, we obtain 
\begin{eqnarray*}
\PP \left(\sigma_1 =k \right) & = & \PP \left(L_{\tau_1+k+1}=1 \Big \vert \tau_1 <\infty\right)\\
& = & \PP \left(L_{\tau_1+k+1}=1 \Big \vert L_{\tau_1 }=0\right)\\
& = & \PP \left(L_{k+1}=1 \Big \vert L_{1}=0\right)=\pi_0^{k-1} (1-\pi_0).
\end{eqnarray*}
Similarly, it follows that $\tau_n \sim Geo((1-\pi_1))$, and $\sigma_n \sim Geo((1-\pi_0))$, where the notation $X \sim Geo(r)$ implies that $\PP\left(X=k\right)= (1-r)^{k-1}r$, for $k \ge 1$. It is immediate that $\E \left[\tau \right]= \frac{1}{1-\pi_1}$ and $\E \left[\sigma \right]= \frac{1}{1-\pi_0}$, and $\E[e^{t\tau}] = \frac{(1-\pi_1)e^t}{1-\pi_1 e^t}$ for $t < - \ln \pi_1$ and $\E[e^{t\sigma}] = \frac{(1-\pi_0)e^t}{1-\pi_0 e^t}$ for $t < - \ln \pi_0$. Choose $\delta = \min (- \ln \pi_1, - \ln \pi_0)$, and thus assumption (\textbf{A2}) is satisfied.
\end{example}
\begin{example}[I.i.d.] If $\{L_n\}_{n \ge 1}$ be i.i.d.\, such that, $\PP(L_n=1)=\pi_1$ where $0 < \pi_1 <1$, and we assume that $L_1 =1$, w.p. $1$. Then it is follows immediately that the inter-arrival times satisfy the assumption (\textbf{A1}).

Furthermore, it follows immediately that $\tau_n \sim Geo((1-\pi_1))$ and $\sigma_n \sim Geo(\pi_1)$, and hence assumption (\textbf{A2}) is satisfied.
\end{example}
\begin{example}[Arithmetic progression]
In the case of the arithmetic progression, we have $\PP(\tau=k)=1$, and
$\PP(\sigma=\mu-k)=1$ for some fixed $\mu, k \in \bN$. Therefore, it is immediate that the assumptions (\textbf{A1}) and (\textbf{A2}) are satisfied.
\end{example}

\section*{Acknowledgment}
The authors are grateful to Vivek Borkar for his valuable insight on the various aspects of stochastic approximation. The authors are grateful to Andrew Wade for the discussion on renewal chains and regenerative processes. The authors would like to thank Rahul Roy for pointing out the appropriate literature on partial memory.

\section*{Grants}
VM acknowledges the financial support of the Deutsche Forschungsgemeinschaft (through Project-ID 444091549 within SPP-2265).

\section*{Appendix A}

In this section, we collect some of the results of the stochastic approximation techniques that we have used so far. In this section, all vectors are written as row vectors to match the notations commonly used in the stochastic approximation literature.
Let $(\theta_n)_{n \ge 1}$ be a sequence of random vectors taking values in $\Rbold^d$, and let it satisfy the following recursion,
\begin{equation}
\label{Eq: recursion_general_SA}
\theta_{n+1}=\theta_{n}+\frac{1}{n+1}\left(h(\theta_n)+ \Delta M_{n+1}+ r_{n+1}\right),
\end{equation}
where $h:\Rbold^d \longrightarrow \Rbold^d$ is called the \textit{mean field} function, $\Delta M_0\equiv 0, \, \left(\Delta M_n\right)_{n \ge0}$ is a martingale difference sequence w.r.t to the filtration $(\FF_n)_{n \ge 1}$, and $r_n$ is the \textit{error} or \textit{remainder} term.

The following theorem is a specialisation of the classical ODE method for stochastic approximation (see Theorem 2.2 in \cite{Borkar_Meyn2000}, \cite{Benaim1996}, or  Chapter 2 in \cite{Borkar}).
\begin{theorem}
\label{Thm: SA_SLLN}
Let $\theta_n$ be as in \eqref{Eq: recursion_general_SA}. Suppose the following assumptions are satisfied. 
\begin{enumerate}[i.]
\item $\sup_{n \ge 1} \lVert \theta_n \rVert < \infty$. 
\item $h$ is a Lipschitz function, that is, there exists $L>0$, such that, $\lVert h(x)-h(y) \rVert \le L \lVert x-y\rVert$.
\item $\sup_{n \ge 1} \E \left[ \lVert \Delta M_{n}\rVert ^2 |\FF_{n-1} \right] < \infty$.
\item Assume that $r_n \longrightarrow 0, \text{ a.s.}$ as $n \to \infty$.
\item Let $\theta^*$ be an equilibrium point of $h$, that is, $\theta^* \in \{\theta: \, h(\theta)=0\}$, such that $\langle \theta-\theta^*, h(\theta)\rangle <0$, for all $\theta \neq \theta^*$.
\end{enumerate}
Then, as $n \to \infty$
\begin{equation}
\label{Eq:SA_SLLN}
\theta_n \longrightarrow \theta^*, \text { a.s.}
\end{equation}
where the convergence is sample path dependent, in case, $\theta^*$ is not unique.
\end{theorem}
The above theorem can be considered as an equivalent for the Strong Law of Large numbers in the field of Stochastic approximation. 
In \cite{Borkar} Corollary 3, pg, 17, the condition v. is stated for $V(\theta)$ where $V(\theta)$ is a Liapunov function, and the equivalent condition is $\langle V(\theta), h(\theta)\rangle <0$. If we choose $V(\theta)= \frac{1}{2} \lVert \theta-\theta^* \rVert ^2$, then we get the above form of v. as in Theorem \ref{Thm: SA_SLLN}.
\begin{corollary}
\label{Cor: SA_SLLN}
If the mean field function can be expressed as $h(\theta) = \theta  \mathbf{A} + \beta$, for some matrix $\mathbf{A}$, then a sufficient condition equivalent to v. in Theorem \ref{Thm: SA_SLLN} is that the eigenvalues of $\mathbf{A}$ are all distinct and negative.   
\end{corollary}

\begin{remark} We would remark here that the recursion for the SA which one can find in equation (1.1) in \cite{Zh_2016}, is written as 
\[
\theta_{n+1}=\theta_{n}+\frac{1}{n+1}\left(-h(\theta_n)+ \Delta M_{n+1}+ r_{n+1}\right).
\]
That is \cite{Zh_2016} writes the mean field function as $-h$, whereas we write it as $h$. Thus, we make the necessary modifications in stating the next few theorems that can be found either explicitly or implicitly in \cite{Zh_2016}. 
\end{remark}
The following theorem is the version Central Limit theorem using SA techniques and is available as Theorem 1.1 in \cite{Zh_2016}.

\begin{theorem}
\label{Thm: CLT_SA_eigenvalue_more_than_half}
 Let $\theta^*$ be an equilibrium point for the mean field function $h$, that is, $\theta^* \in \{\theta: \, h(\theta)=0\}$, where $h$ is as in \eqref{Eq: recursion_general_SA}. Let $\theta_n$, $\Delta M_n$, and $r_n$ be as in \eqref{Eq: recursion_general_SA}, and $\theta_n \longrightarrow \theta^*, \text{ a.s.}$ as $n \to \infty$. Let $\delta>0$ be such that,
\begin{equation}\label{eq: SA CLT cond1 rho>1/2}
    \sup_{n \ge 1}\E \left[\lvert \lvert \Delta M_{n+1}\rvert \rvert ^{2+\delta} | \ \FF_{n}\right]< \infty, \text{ a.s.}
\end{equation}
Furthermore, as $n \to \infty$
\begin{equation}\label{eq: SA CLT cond2 rho>1/2}
\E \left[ \Delta M_n \left( \Delta M_n \right)^\top | \ \FF_{n}\right] \longrightarrow \Gamma, \text{ a.s.}
\end{equation}
where $\Gamma$ is a deterministic symmetric positive semi-definite matrix and for any $\epsilon>0$, 
\begin{equation}
\label{Eq: error_for_second_eigenvalue_leq_1/2}
(n+1)\E\left[\lvert \lvert r_{n+1} \rvert \rvert ^2 \mathbf{1}_{\{\lvert \lvert \theta_n - \theta^*\rvert \rvert \le \epsilon\}} \right] \longrightarrow 0.
\end{equation}
Let us assume that the eigenvalues of $D(-h(\theta^*))$ are all real and positive. Let $\rho^{*}$ be the smallest eigenvalue. If $\rho^{*} >1/2$, then as $n \to \infty$ 
\begin{equation}
\label{Eq: CLT_second_eigenvalue_leq_1/2}
\sqrt{n}\left(\theta_n - \theta^*\right)\stackrel{d}{\longrightarrow} N(0, \Sigma),
\end{equation}
where the variance-covariance matrix $\Sigma$ is given by 
\begin{equation}
\label{Eq: variance_covariance_matrix_eigenvalue_more_than_half}
\Sigma = \int_{0}^{\infty}{\left(e^{-(D(-h(\theta^*)-\frac{1}{2}I_d) u}\right)^\top \Gamma e^{-(D(-h(\theta^*)-\frac{1}{2}I_d) u} \mathrm{d}u},
\end{equation}
where $I_d$ denotes the $d$-dimensional identity matrix.
\end{theorem}

\begin{remark}
In the statement of Theorem 1.1 in \cite{Zh_2016}, the eigenvalues of $D(-h(\theta^*))$ can be complex with positive real parts. However, since in all our application of the above theorem the eigenvalues are all real, we state the above theorem for real eigenvalues only. 
\end{remark}
Next, we state the version of the Central Limit Theorem, when the second largest eigenvalue $\rho^*= \frac{1}{2}$.
\begin{theorem}
\label{Thm: CLT_SA_for_rho_equals_1/2}
Let $\theta^*$ be an equilibrium point for the mean field function $h$, that is, $\theta^* \in \{\theta: \, h(\theta)=0\}$, where $h$ is as in \eqref{Eq: recursion_general_SA}. Let $\theta_n$, $\Delta M_n$, and $r_n$ be as in \eqref{Eq: recursion_general_SA}, and $\theta_n \longrightarrow \theta^*, \text{ a.s.}$ as $n \to \infty$. 
Further, assume that 
\begin{equation}
\label{Eq: Taylor_expnasion_for_theta}
h(\theta)= \left(\theta -\theta^*\right) Dh(\theta^*)+ o(\lvert \lvert \theta -\theta^*\rvert \rvert^{1+\epsilon}) \text{ as } \theta \rightarrow \theta^*, \text{ for some } \epsilon>0,
\end{equation}
and the following version of the Lindeberg-Feller condition holds
\begin{equation}
\label{Eq: Decay_martingale_difference_rho_half}
\frac{1}{n} \sum_{m=1}^n \E\left[\lvert \lvert \Delta M_m\rvert \rvert ^2 \mathbf{1}_{\{\lvert \lvert \Delta M_m\rvert \rvert \ge \epsilon \sqrt{n}\}}\mid  \FF_{m-1}\right] \longrightarrow 0, \text{ a.s. or in } L_1, \text{ for all } \epsilon>0.
\end{equation}
Let $\Gamma$ be a positive semi-definite matrix, such that,
\begin{equation}
\label{Eq: Gamma_as_limit}
\frac{1}{n} \sum_{m=1}^n \E \left[\Delta M_m \left(\Delta M_m\right)^\top \mid \FF_{m-1}\right]\rightarrow \Gamma, \text { a.s. or in } L_1. 
\end{equation}
Let the maximum eigenvalue $\rho^*=\frac{1}{2}$, and the error $r_n$ satisfy the following 
\begin{equation}
\label{Eq: sum_of_rn}
\sum_{k=1}^n r_k = o\left(\sqrt{\frac{n}{log n}}\right), \text{ a.s.}
\end{equation}
Then 
\begin{equation}
\label{Eq: CLT_rho_half}
\frac{\sqrt{n}}{\sqrt{\log n}} \left(\theta_n -\theta^*\right)\stackrel{d}{\longrightarrow} N\left(0, \widetilde{\Sigma}\right) \text{(stably)},
\end{equation}
where 
\begin{equation}
\label{Eq: Description_tilde_Sigma}
\widetilde{\Sigma}= \lim_{n \to \infty}\frac{1}{\log n}\int_{0}^{\log n} {\left(e^{-(D(-h(\theta^*)-\frac{1}{2}I_d) u}\right)^\top \Gamma e^{-(D(-h(\theta^*)-\frac{1}{2}I_d) u} \mathrm{d}u}
\end{equation}
\end{theorem}
\begin{remark}A more general version of the above theorem is available in Theorem 2.1 in \cite{Zh_2016} when $Dh(\theta^*)$ is not necessarily diagonalisable. However, since in our case $Dh(\theta^*)$ is diagonalisable, the version of the corresponding Central limit theorem in our case is simplified. {If $Dh(\theta^*)$ is not necessarily diagonalisable, for example, when the eigenvalues are repeated, then the idea is to represent $Dh(\theta^*)$ in its Jordan canonical form.} In our case, the quantity $\nu=1$ since the eigenvalues of $Dh(\theta^*)$ are all distinct, real and non-zero, where $\nu$ is as defined in Theorem 2.1 in \cite{Zh_2016}. In this case, all eigenvalues are real but not distinct, $\nu$ is the algebraic multiplicity of the eigenvalue $\rho^*$. For a more detailed description of $\widetilde{\Sigma}$ in the general case, we refer the reader to Theorem 2.1 in \cite{Zh_2016}.
\end{remark}
In the case, when the second largest eigenvalue $\rho^*< \frac{1}{2}$, it can be shown that with suitable scaling which is of the order $n^{\rho^*}$ the fluctuations of $\theta_n$ converge almost surely to a suitable random variable. The following theorem is available as Theorem 2.2 in \cite{Zh_2016}.

\begin{theorem}
\label{Thm: CLT_SA_for_rho_leq_1/2} Let $\theta^*$ be an equilibrium point for the mean field function $h$, that is, $\theta^* \in \{\theta: \, h(\theta)=0\}$, where $h$ is as in \eqref{Eq: recursion_general_SA}. Let $\theta_n$, $\Delta M_n$, and $r_n$ be as in \eqref{Eq: recursion_general_SA}, and $\theta_n \longrightarrow \theta^*, \text{ a.s.}$ as $n \to \infty$. 
Further, assume that 
\begin{equation}
\label{Eq: Taylor_expnasion_for_theta_1}
{h(\theta)= \left(\theta -\theta^*\right) Dh(\theta^*)+ o(\lvert \lvert \theta -\theta^*\rvert \rvert^{1+\epsilon}) \text{ as } \theta \rightarrow \theta^*, \text{ for some } \epsilon>0.}
\end{equation}
Let $0< \rho^* < 1/2$. Furthermore, assume that 
\begin{equation}
\label{Eq:_second_moment_martingale_difference_rho_leq_half}
\sum_{m=1}^n \E \left[\Delta M_m \left(\Delta M_m\right)^\top \mid \FF_{m-1}\right]= \OO(n), \text{ a.s. or in } L_1,
\end{equation}
and 
\begin{equation}
\label{Eq: error_rate_rho_leq_half}
\sum_{k=1}^n r_k= o\left(n^{1-\rho^*-\delta_0}\right)
\text{ a.s. for some } \delta_0>0.
\end{equation}
Then there exists a random variable $L$, such that, 
\begin{equation}
\label{Eq: CLT_rho_leq_half}
n^{\rho^*}\left(\theta_n-\theta^*\right)\rightarrow L, \text{ a.s.}
\end{equation}
\end{theorem}
\begin{remark} 
\label{Rem: discussion_rho*_leq_half}
Again as in the previous cases, the version of the theorem that we see as Theorem 2.2 in \cite{Zh_2016} is written for a much more general case, when the matrix $Dh(\theta^*)$ is not necessarily diagonalisable. In this general case, the scaling and also the limiting random variable will depend on the algebraic multiplicity of the second largest eigenvalue of $Dh(\theta^*)$, when all the eigenvalues are real. 
\end{remark}

\section*{Appendix B}\label{Appendix B}

This section contains the calculations for $\Sigma$ and $\widetilde{\Sigma}$ from Proposition \ref{prop: limiting variance}.
We first compute the quantity $e^{\left((A+\frac{1}{2} I)u\right)^\top} \Gamma e^{(A+\frac{1}{2} I )u}$.
\ 

\noindent Let $B \coloneqq A+\frac12 I= \begin{pmatrix}
\alpha & 0\\
\beta & -\frac12
\end{pmatrix}$,
where
\[
\alpha=(2p-1)\frac{k}{\mu}-\frac12,
\qquad
\beta=(2p-1)\frac{\mu-k}{\mu}.
\]
Since $B$ is lower triangular, we get
\[
B^n=
\begin{pmatrix}
\alpha^n & 0  \\ 
\beta\sum_{j=0}^{n-1}\alpha^{\,n-1-j} \left(-\frac12\right)^j & \left(-\frac12\right)^n
\end{pmatrix} 
= 
\begin{pmatrix}
\alpha^n & 0\\
\beta \frac{\alpha^n-\left(-\frac12\right)^n}
{\alpha+\frac12} & \left(-\frac12\right)^n
\end{pmatrix}.
\]
Then, 
\begin{align*}
e^{Bu}
&= \sum_{n=0}^{\infty} \frac{u^n}{n!}B^n \\
&= \begin{pmatrix}
\sum_{n=0}^{\infty}\frac{(\alpha u)^n}{n!} & 0 \\
\frac{\beta}{\alpha+\frac12}\sum_{n=1}^{\infty}\frac{u^n}{n!} \left(\alpha^n-\left(-\frac{1}{2}\right)^n\right)
& \sum_{n=0}^{\infty}\frac{(-u/2)^n}{n!}
\end{pmatrix}. \\
&= \begin{pmatrix}
e^{\alpha u} & 0 \\
\beta \dfrac{e^{\alpha u}-e^{-u/2}}{\alpha+\frac{1}{2}} & e^{-u/2}
\end{pmatrix}.
\end{align*}
We get,
\[ e^{B^\top u} \Gamma e^{Bu} = \begin{pmatrix}
    ke^{2\alpha u} + (\mu-k)f(u)^2 & (\mu-k)f(u) e^{-u/2} \\
    (\mu-k)f(u)e^{-u/2} & (\mu-k)e^{-u}
\end{pmatrix},\]
where $f(u) = \beta \dfrac{e^{\alpha u}-e^{-u/2}}{\alpha + 1/2} = \frac{\mu - k}{k}(e^{\alpha u}-e^{-u/2})$. Therefore,
\begin{align*} 
e^{B^\top u} \Gamma e^{Bu} &= (1-\eta^2) \begin{pmatrix}
    ke^{2\alpha u} + \frac{(\mu-k)^3}{k^2}(e^{\alpha u}-e^{-u/2})^2 & \frac{(\mu-k)^2}{k} (e^{\alpha u}-e^{-u/2}) e^{-u/2} \\
    \frac{(\mu-k)^2}{k}(e^{\alpha u}-e^{-u/2})e^{-u/2} & (\mu-k)e^{-u}
\end{pmatrix} \\
& =
(1-\eta^2)\begin{pmatrix}
\left( k+\frac{(\mu-k)^3}{k^2}
\right)e^{2\alpha u} -\frac{2(\mu-k)^3}{k^2} e^{(\alpha-\frac{1}{2})u} +\frac{(\mu-k)^3}{k^2}e^{-u}
&
\frac{(\mu-k)^2}{k} \left(e^{(\alpha-\frac{1}{2})u} -e^{-u}
\right) \\
\frac{(\mu-k)^2}{k}
\left(e^{(\alpha-\frac{1}{2})u}
-e^{-u} \right)
&
(\mu-k)e^{-u}
\end{pmatrix}.
\end{align*}
\begin{enumerate}[I.]
    \item When $(2p-1)\frac{k}{\mu} < \frac{1}{2}$, we want to compute 
    \[ \int_0^\infty \left[ e^{\left((A+\frac{1}{2} I)u\right)^\top} \Gamma e^{\left((A+\frac{1}{2} I )u \right)} \right] du. \]
With this condition, we have
\begin{align*}
\int_0^\infty e^{2\alpha u}\,du &= -\frac{1}{2\alpha} = \frac{1}{1-2(2p-1)\frac{k}{\mu}},\\
\int_0^\infty e^{(\alpha-\frac12)u}\,du &= -\frac{1}{\alpha-\frac12} = \frac{1}{1-(2p-1)\frac{k}{\mu}},\\
\int_0^\infty e^{-u}\,du &=1.
\end{align*}
Thus, 
\begin{align*}
\sigma_{11} &= (1-\eta^2)\left[
\frac{k+\dfrac{(\mu-k)^3}{k^2}}{1-2(2p-1)\dfrac{k}{\mu}} - \frac{2(\mu-k)^3}{k^2\left(1-(2p-1)\dfrac{k}{\mu}\right)}+\frac{(\mu-k)^3}{k^2}
\right],\\
\sigma_{12} &= (1-\eta^2) \frac{(\mu-k)^2}{k} \left(\frac{1}{1-(2p-1)\dfrac{k}{\mu}} -1\right),\\
\sigma_{22} &= (1-\eta^2)(\mu-k). 
\end{align*}
Thus, 
\[\Sigma = \sigma_{11} + 2\sigma_{12}+\sigma_{22} = (1-\eta^2)\left(\frac{k+ \dfrac{(\mu-k)^3}{k^2}}{1-2(2p-1)\dfrac{k}{\mu}} + \frac{\frac{2(\mu-k)^2}{k}-\frac{2(\mu-k)^3}{k^2}}{\left(1-(2p-1)\dfrac{k}{\mu}\right)}+\frac{(\mu-k)^3}{k^2} - \frac{2(\mu-k)^2}{k} + \mu-k \right).\]
\item When $(2p-1)\frac{k}{\mu} = \frac{1}{2}$, we want to compute 
    \[ \lim_{t \to \infty} \frac{1}{\log t} \int_0^{\log t} \left[ e^{\left((A+\frac{1}{2} I)u\right)^\top} \Gamma e^{\left((A+\frac{1}{2} I )u \right)} \right] du. \]
    Since $\alpha = 0$ in this case, we have
\begin{equation*}
    e^{B^\top u} \Gamma e^{Bu} =
(1-\eta^2)\begin{pmatrix}
 k+\frac{(\mu-k)^3}{k^2} -\frac{2(\mu-k)^3}{k^2} e^{-\frac{1}{2}u} +\frac{(\mu-k)^3}{k^2}e^{-u}
&
\frac{(\mu-k)^2}{k} \left(e^{-\frac{1}{2}u} -e^{-u}
\right) \\
\frac{(\mu-k)^2}{k}
\left(e^{-\frac{1}{2}u}
-e^{-u} \right)
&
(\mu-k)e^{-u}
\end{pmatrix}.
\end{equation*}
Observe that \[\lim_{t\rightarrow\infty} \frac{1}{\log t} \int_0^{\log t}  e^{-u}du = \lim_{t\rightarrow\infty} \frac{1-\frac{1}{t}}{\log t} = 0. \]
Thus, $\sigma_{11} = k+\frac{(\mu-k)^3}{k^2} $ and $\sigma_{12} = \sigma_{21} = \sigma_{22} = 0$ and hence $\widetilde{\Sigma} = k+\frac{(\mu-k)^3}{k^2}$.
\end{enumerate}
\newpage
\section*{Appendix C}\label{Appendix C}
In this section, we include the plots obtained by simulations in the case where the size of $D_n$ is trivial. Recall that the CLT-type results in Theorem \ref{Thm: CLT_D_n/n_goes_to_zero} are for  $|D_n| = o(\sqrt{n})$. The following plots show the behaviour of TMERW when $|D_n|$ is between $\sqrt{n}$ and $n$. We plot histograms and the respective q-q plots for 1000 realisations of the walk for $p=0.35 < \frac{3}{4}$ and $|D_n| = \lfloor n ^\gamma\rfloor$ for two different values of $\gamma$.
\begin{figure}[h!]
    \centering
        \begin{minipage}{0.47\textwidth}
        \centering
        \includegraphics[width=\textwidth]{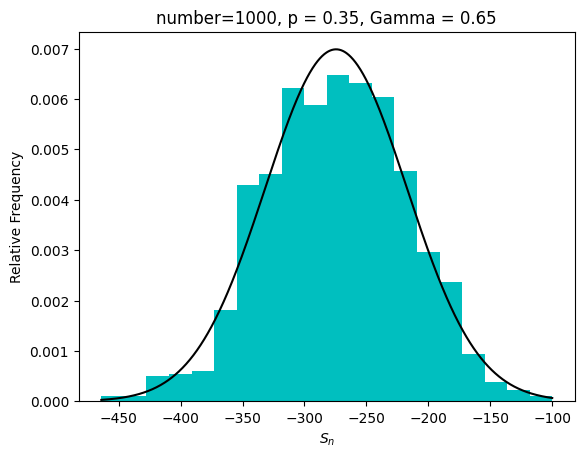}
        \caption{$|D_n| =\OO(n^{0.65})$}
        \label{fig:img2}
    \end{minipage}
    \begin{minipage}{0.47\textwidth}
        \centering
\includegraphics[width=\textwidth]{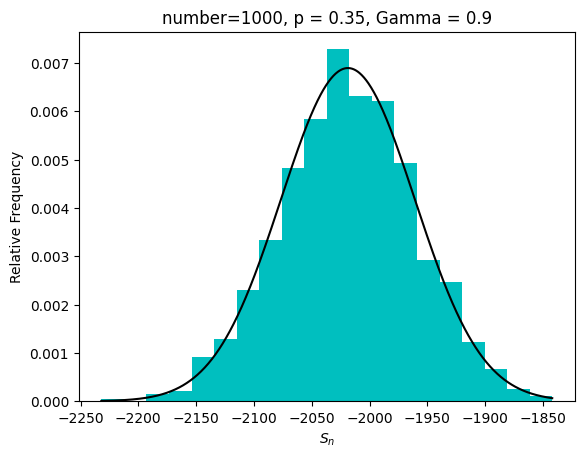}
        \caption{$|D_n| =\OO(n^{0.9})$}
        \label{fig:img1}
    \end{minipage}
    \hfill 

\end{figure}
\begin{figure}[h!]
    \centering
     \begin{minipage}{0.47\textwidth}
        \centering
        \includegraphics[width=\textwidth]{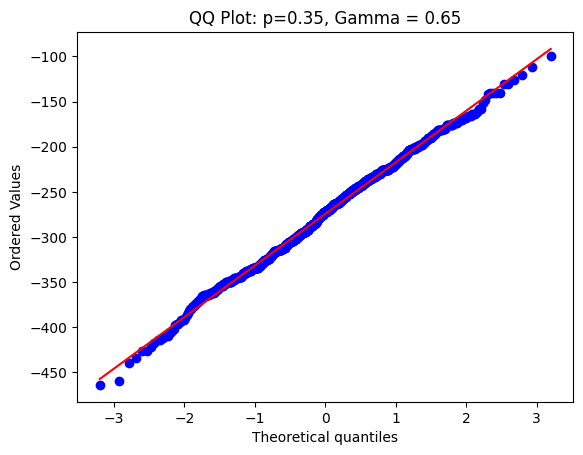}
        \caption{$|D_n| =\OO(n^{0.65})$}
        \label{fig:img4}
    \end{minipage}
    \begin{minipage}{0.47\textwidth}
        \centering
\includegraphics[width=\textwidth]{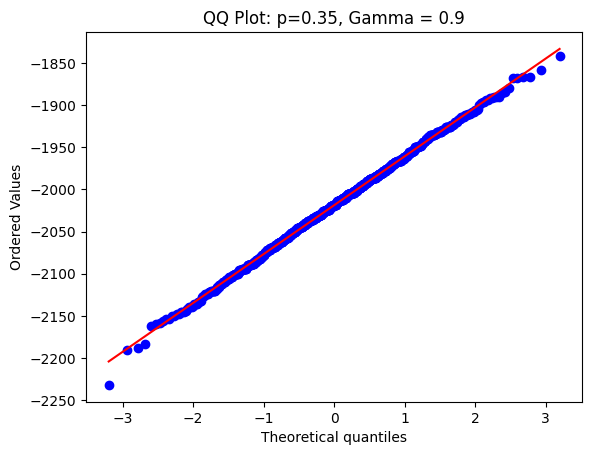}
        \caption{$|D_n| =\OO(n^{0.9})$}
        \label{fig:img3}
    \end{minipage}
    \hfill 
   
\end{figure}

\bibliographystyle{plain}
\bibliography{ref}
\end{document}